\documentclass[11pt]{amsart}
\usepackage{geometry}                
\usepackage{graphicx}
\usepackage{amssymb}
\usepackage{epstopdf}
\usepackage{geometry} 
\usepackage{graphicx}
\usepackage{amsmath} 
\numberwithin{equation}{section}
\usepackage{commath}
\usepackage{enumitem}
\usepackage{changepage}
\usepackage{cleveref}
\usepackage[normalem]{ulem}
\usepackage{bbm}
\usepackage{subfig}
\usepackage{float}
\usepackage{url}
\usepackage{mathtools}
\usepackage{amsthm} 
\usepackage[disable]{todonotes} 
\DeclareMathOperator*{\esssup}{ess\,sup}

\theoremstyle{plain}
\newtheorem{theorem}{Theorem}[section]
\newtheorem*{theorem*}{Theorem}
\newtheorem{lemma}[theorem]{Lemma}
\newtheorem{corollary}[theorem]{Corollary}
\newtheorem{proposition}[theorem]{Proposition}

\theoremstyle{remark}
\newtheorem*{remark}{Remark}
\theoremstyle{definition}
\newtheorem{definition}[theorem]{Definition}
\allowdisplaybreaks
\title[Stability for piecewise-smooth solutions]{Criteria for the a-contraction and stability for the piecewise-smooth solutions to hyperbolic balance laws}
\author[Krupa]{Sam G. Krupa}
\address{Department of Mathematics\\ The University of Texas at Austin\\ Austin, TX 78712\\ USA}
\email{skrupa@math.utexas.edu}
\thanks{This work was partially supported by NSF Grant DMS-1614918.}
\date{April 20th, 2019}                                           

\begin{document}
\begin{abstract}
We show uniqueness and stability in $L^2$ and for all time for piecewise-smooth solutions to hyperbolic balance laws. We have in mind applications to gas dynamics, the isentropic Euler system and the  full Euler system for a polytropic gas in particular. We assume the discontinuity in the piecewise smooth solution is an extremal shock. We use only mild hypotheses on the system. Our techniques and result hold without smallness assumptions on the solutions. We can handle shocks of any size. We work in the class of bounded, measurable solutions satisfying a single entropy condition. We also assume a strong trace condition on the solutions, but this is weaker than $BV_{\text{loc}}$. We use the theory of a-contraction (see Kang and Vasseur [{\em Arch. Ration. Mech. Anal.}, 222(1):343--391, 2016]) developed for the stability of pure shocks in the case without source.
\end{abstract}
\maketitle
\tableofcontents

\section{Introduction}

We consider an $n\times n$ system of balance laws,
\begin{align}
\label{system}
\begin{cases}
\partial_t u + \partial_x f(u)=G(u(\cdot,t))(x),\mbox{ for } x\in\mathbb{R},\mbox{ } t>0,\\
u(x,0)=u^0(x) \mbox{ for } x\in\mathbb{R}.
\end{cases}
\end{align}

For a fixed $T>0$ (including possibly $T=\infty$), the \emph{unknown} is $u\colon\mathbb{R}\times[0,T)\to \mathbb{M}^{n\times 1}$. The function $u^0\colon\mathbb{R}\to\mathbb{M}^{n\times 1}$ is in $L^\infty(\mathbb{R})$ and is the \emph{initial data}. The function $f\colon\mathbb{M}^{n\times 1}\to\mathbb{M}^{n\times 1}$ is the \emph{flux function} for the system. The \emph{source term} $G\colon (L^2(\mathbb{R}))^n\to (L^2(\mathbb{R}))^n$ is translation invariant. We also ask that $G$ be Lipschitz continuous from $(L^2(I))^n\to (L^2(I))^n$ for every interval $I\subseteq\mathbb{R}$, with a Lipschitz constant uniform in $I$. In other words, there exists $C_G>0$ such that
\begin{align}\label{G_acts_like}
\norm{G(g_1)-G(g_2)}_{L^2(I)}\leq C_G \norm{g_1-g_2}_{L^2(I)},
\end{align}
for every $g_1,g_2\in(L^2(\mathbb{R}))^n$ and for every interval $I\subseteq\mathbb{R}$.  Furthermore, we require that $G$ is bounded on $(L^\infty(\mathbb{R}))^n$:
\begin{align}\label{G_acts_like_2}
\norm{G(g)}_{L^\infty(\mathbb{R})}\leq C_G \norm{g}_{L^\infty(\mathbb{R})},
\end{align}
for every $g\in(L^\infty(\mathbb{R}))^n$. 

We assume the system \eqref{system} is endowed with a strictly convex entropy $\eta$ and associated entropy flux $q$. Note the system will be hyperbolic on the state space where $\eta$ exists. We assume the functions $f, \eta$, and $q$ are defined on an open convex state space $\mathcal{V}\subset\mathbb{R}^n$. We assume $f,q\in C^2(\mathcal{V})$ and $\eta \in C^3(\mathcal{V})$.  By assumption, the entropy $\eta$ and its associated entropy flux $q$ verify the following compatibility relation:
\begin{align}\label{compatibility_relation_eta_system}
\partial_j q =\sum_{i=1}^n \partial_i\eta\partial_j f_i,\hspace{.25in} 1\leq j \leq n.
\end{align}
By convention, the relation \eqref{compatibility_relation_eta_system} is rewritten as 
\begin{align}
\nabla q = \nabla \eta \nabla f,
\end{align}
where $\nabla f$ denotes the matrix $(\partial_j f_i)_{i,j}$.

For $u\in\mathcal{V}$ where $\eta$ exists , the system \eqref{system} is hyperbolic, and the matrix $\nabla f(u)$ is diagonalizable, with eigenvalues 
\begin{align}
\lambda_1(u)\leq \ldots \leq \lambda_n(u),
\end{align}
called \emph{characteristic speeds}.

We consider both bounded  \emph{classical} and bounded \emph{weak} solutions to \eqref{system}. A weak solution $u$ is bounded and measurable and satisfies \eqref{system} in the sense of distributions. I.e., for every Lipschitz continuous test function $\Phi:\mathbb{R}\times[0,T)\to \mathbb{M}^{1\times n}$ with compact support,
\begin{equation}
\begin{aligned}\label{u_solves_equation_integral_formulation_chitchat}
\int\limits_{0}^{T} \int\limits_{-\infty}^{\infty} \Bigg[\partial_t\Phi u + \partial_x\Phi f(u) \Bigg]\,dxdt +\int\limits_{-\infty}^{\infty} \Phi(x,0)u^0(x)\,dx
\\
=-\int\limits_{0}^{T} \int\limits_{-\infty}^{\infty}\Phi G(u(\cdot,t))(x)\,dxdt.
\end{aligned}
\end{equation}

We only consider solutions $u$ which are entropic for the entropy $\eta$. That is, they satisfy the following entropy condition:
\begin{align}\label{entropy_condition_distributional_system_chitchat}
\partial_t \eta(u)+\partial_x q(u) \leq \nabla\eta(u)G(u(\cdot,t))(x),
\end{align}
in the sense of distributions. I.e., for all positive, Lipschitz continuous test functions $\phi:\mathbb{R}\times[0,T)\to\mathbb{R}$ with compact support:
 \begin{equation}
\begin{aligned}\label{u_entropy_integral_formulation_chitchat}
\int\limits_{0}^{T} \int\limits_{-\infty}^{\infty}\Bigg[\partial_t\phi\big(\eta(u(x,t))\big)+&\partial_x \phi \big(q(u(x,t))\big)\Bigg]\,dxdt+ \int\limits_{-\infty}^{\infty}\phi(x,0)\eta(u^0(x))\,dx\geq
\\
&-\int\limits_{0}^{T} \int\limits_{-\infty}^{\infty}\phi\nabla\eta(u(x,t))G(u(\cdot,t))(x)\,dxdt.
\end{aligned}
\end{equation}

For $u_L,u_R\in\mathbb{R}^n$, the function $u:\mathbb{R}\times[0,\infty)\to\mathbb{R}^n$ defined by
\begin{align}\label{shock_solution_system}
u(x,t)\coloneqq
\begin{cases}
u_L &\mbox{ if } x<\sigma t ,\\
u_R &\mbox{ if } x>\sigma t
\end{cases}
\end{align}
is a weak solution to \eqref{system} if and only if  $u_L,u_R$, and $\sigma$ satisfy the Rankine-Hugoniot jump compatibility relation:
\begin{align}\label{RH_jump_condition}
f(u_R)-f(u_L)=\sigma (u_R-u_L),
\end{align}
in which case \eqref{shock_solution_system} is called a \emph{shock} solution.

Moreover, the solution \eqref{shock_solution_system} will be entropic for $\eta$ (according to \eqref{u_entropy_integral_formulation_chitchat})  if and only if,
\begin{align}\label{entropic_shock_condition_system}
q(u_R)-q(u_L)\leq \sigma (\eta(u_R)-\eta(u_L)).
\end{align}
In this case, $(u_L,u_R,\sigma)$ is an \emph{entropic Rankine--Hugoniot discontinuity}.

For a fixed $u_L$, we consider the set of $u_R$ which satisfy \eqref{RH_jump_condition} and \eqref{entropic_shock_condition_system} for some $\sigma$. For a general $n\times n$ strictly hyperbolic system of conservation laws endowed with a strictly convex entropy , we know that locally this set of $u_R$ values is made up of $n$ curves (see for example \cite[p.~140-6]{lefloch_book}).

The present paper concerns the finite-time stability of piecewise-smooth solutions to \eqref{system}, working in the $L^2$ setting. We work in a very general setting. Our techniques are based on the theory of shifts as developed by Vasseur within the context of the relative entropy method (see \cite{VASSEUR2008323}). We consider systems of the form \eqref{system}, with minimal assumptions on the shock families. We ask that the extremal shock speeds (1-shock and n-shock speeds) are separated from the intermediate shock families. If we want to consider 1-shocks, we ask that the 1-shock family satisfy the Liu entropy condition (shock speed decreases as the right-hand state travels down the 1-shock curve), and we ask that the shock strength increase in the sense of relative entropy (an $L^2$ requirement) as the right-hand state travels down the 1-shock curve. If we want to consider n-shocks, we ask for similar requirements on the n-shock family. 

 The intermediate wave families have far fewer requirements. The intermediate shock curves might not even be well-defined and characteristic speeds might cross.
 
In particular, the results in this article apply to both the isentropic Euler system and the  full Euler system for a polytropic gas, viewing both systems in Eulerian coordinates.

We study solutions $\bar{u}$ which are piecewise-Lipschitz continuous in the space variable $x$. We study the stability and uniqueness of these solutions among a large class of weak solutions $u$ which are bounded, measurable, entropic for at least one strictly convex entropy, and verify a strong trace condition (weaker than $BV_{\text{loc}}$). We do not make small data assumptions. We require the piecewise-smooth  $\bar{u}$ contain a single shock of extremal family. However, the rougher solutions $u$ which we compare to this solution $\bar{u}$ may have shocks of any type or family.

Previous results in the theory of stability and a-contraction have only been able to consider initial data which is pure shock (piecewise constant).  This present paper extends the ideas in the theory of a-contraction (in particular as developed in \cite{MR3519973}).

\begin{figure}[tb]
      \includegraphics[width=\textwidth]{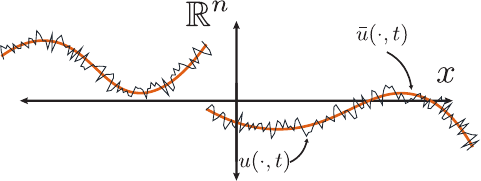}
  \caption{In this paper, we study the stability of solutions $u$ (to \eqref{system}) which are $L^2$ perturbations of a piecewise-smooth solution $\bar{u}$, as shown in this schematic. The nonlinearity in the solution $\bar{u}$ causes significant technical challenges not present in the piecewise-constant case (for the piecewise-constant case, see \cite{Leger2011,MR3519973}).}\label{figure_diagram}
\end{figure}

The point of the present article is this: As discussed for the case of nonlocal scalar balance laws in \cite{scalar_move_entire_solution}, when studying the stability up to a translation in space of solutions piecewise-constant in space, we can view the shift function which is doing the translation as simply determining at which points do we want to see the left hand state of our solution, and at which points do we want to see the right hand state of our solution. However, for piecewise-\emph{smooth} data, the shift function cannot be viewed like this. Instead, the shift function is viewed as artificially translating in space our solution. If the solution is non-constant away from the discontinuity, this artificial translation creates a linear term in the entropy dissipation (see \Cref{local_entropy_dissipation_rate_systems}), which we cannot Gronwall in comparison with the quadratic terms. The answer is to create a shift function which not only neutralizes entropy production at the discontinuity of the solution, but also creates  additional negative entropy (see \Cref{systems_entropy_dissipation_room}) we can use to cancel out the linear term in the Gronwall argument (see \Cref{figure_diagram}). Regarding the idea of additional negative entropy caused by a shift, see \cite{2017arXiv171207348K}.

This work is related to the \emph{generalized Riemann problem}, which concerns solutions with initial data which is piecewise-smooth instead of simply piecewise-constant across a single jump discontinuity. For existence and uniqueness results for the generalized Riemann problem, see \cite{MR2070131,MR1680921}. However, these results have small data limitations.

Previous results in this direction include Chen, Frid, and Li \cite{MR1911734} where for the full Euler system,  they show uniqueness and long-time stability for perturbations of Riemann initial data among a large class of entropy solutions (locally $BV$ and without smallness conditions) for the $3\times3$ Euler system in Lagrangian coordinates. They also show uniqueness for solutions piecewise-Lipschitz in $x$. For an extension to the relativistic Euler equations, see Chen and Li \cite{MR2068444}. However, these papers do not give $L^2$ stability results for all time.

We study the stability in $L^2$ of piecewise-smooth solutions to the system of balance laws \eqref{system}.  The study of piecewise-smooth solutions takes us a step beyond the classical Riemann problem, which considers piecewise-constant initial data. Furthermore, when the system \eqref{system} has the source term $G$, it is important to study piecewise-smooth solutions and just not piecewise-constant, for the source term may mean that even pure shock wave initial data evolves into something more complicated. For a nonlocal example of this phenomenon, consider the solution to the Riemann problems for the Burgers--Hilbert equation, which is Burgers equation with a nonlocal source term \cite{MR3248030,MR3605552,MR3348783,MR2982741}.

Our method is the relative entropy method, a technique created by Dafermos \cite{doi:10.1080/01495737908962394,MR546634} and DiPerna \cite{MR523630}  to give $L^2$-type stability estimates between a Lipschitz continuous solution and a rougher solution, which is only weak and entropic for a strictly convex entropy (the so-called \emph{weak-strong} stability theory). For a system \eqref{system} endowed with an entropy $\eta$, the technique of relative entropy considers the quantity called the
\emph{relative entropy}, defined as
\begin{align}
\eta(u|v)\coloneqq \eta(u)-\eta(v)-\nabla\eta(v)\cdot (u-v).
\end{align}

Similarly, we define relative entropy-flux,
\begin{align}
q(u;v)\coloneqq q(u)-q(v)-\nabla\eta(v)\cdot (f(u)-f(v)).
\end{align}

Remark that for any constant $v\in\mathbb{R}^n$, the map $u\mapsto\eta(u|v)$ is an entropy for the system \eqref{system}, with associated entropy flux $u\mapsto q(u;v)$. Furthermore, if $u$ is a weak solution to \eqref{system} and entropic for $\eta$, then $u$ will also be entropic for $\eta(\cdot|v)$. This can be calculated directly from \eqref{system} and \eqref{entropy_condition_distributional_system_chitchat} -- note that the map $u\mapsto\eta(u|v)$ is basically $\eta$ plus a linear term.

Moreover, by virtue of $\eta$ being \emph{strictly} convex, the relative entropy is comparable to the $L^2$ distance, in the following sense:

\begin{lemma}\label{entropy_relative_L2_control_system} For any fixed compact set $V\subset\mathcal{V}$, there exists  $c^*,c^{**}>0$ such that for all $u,v\in V$,
\begin{align}
c^*\abs{a-b}^2\leq \eta(u|v)\leq c^{**}\abs{a-b}^2.
\end{align}
The constants $c^*,c^{**}$ depend on $V$ and bounds on the second derivative of $\eta$.
\end{lemma}

This lemma follows from Taylor's theorem; for a proof see \cite{Leger2011,VASSEUR2008323}.

Given a Lipschitz solution $\bar{u}$ to \eqref{system}, and a weak, entropic solution $u$, the method of relative entropy gives estimates on the growth in time of the quantity
\begin{align*}
\norm{\bar{u}(\cdot,t)-u(\cdot,t)}_{L^2(\mathbb{R})}
\end{align*}
by studying the time derivative $\partial_t\int\eta(u|\bar{u})\,dx$ and using the entropy inequality \eqref{entropy_condition_distributional_system_chitchat}. By \eqref{entropy_relative_L2_control_system}, we get $L^2$-type stability estimates.

Introducing a discontinuity into $\bar{u}$ causes difficulties in the method of relative entropy. In particular, simple examples for the scalar conservation laws show that a discontinuity in $\bar{u}$ prevents stability between $\bar{u}$ and $u$ in the form of the classical weak-strong estimates.

However, by allowing the discontinuity in $\bar{u}$ to move with an artificial speed which depends on $u$, we can recover weak-strong type estimates. Within the context of the relative entropy method, this theory of stability up to a shift was initiated in \cite{VASSEUR2008323} by Vasseur. Over the last decade, this theory of stability up to a shift has been matured and developed by Vasseur and his team. The first result was for pure shock wave initial data for the scalar conservation laws \cite{Leger2011_original}.  Further results include work on the scalar viscous conservation laws in both one space dimension \cite{MR3592682} and multiple \cite{multi_d_scalar_viscous_9122017}. Recently, work on the scalar conservation laws has allowed for many discontinuities to exist in the otherwise smooth $\bar{u}$ -- with each discontinuity shifted in such a way as to maintain $L^2$ stability between $\bar{u}$ and an arbitrary weak solution $u$ entropic for at least one entropy. With this, it is possible to make comparisons between two solutions which satisfy only one entropy condition, and thus show that one entropy condition is enough for uniqueness. See \cite{2017arXiv170905610K} (and the references therein) for more details. To study the $L^2$ stability of pure shock wave initial data in the systems case, the technique of a-contraction was introduced \cite{MR3519973,MR3479527,MR3537479,serre_vasseur,Leger2011}. For a general overview of theory of shifts and the relative entropy method, see \cite[Section 3-5]{MR3475284}. By considering stability up to a shift, the method of relative entropy can also be used to study the asymptotic limit when the limit is discontinuous (see \cite{MR3333670} for the scalar case, \cite{MR3421617} for systems).  There is a long history of using the relative entropy method to study the asymptotic limit. However, without the theory of shifts, it appears that only limits which are Lipschitz continuous can be studied (see  \cite{MR1842343,MR1980855,MR2505730,MR1115587,MR1121850,MR1213991,MR2178222,MR2025302} and \cite{VASSEUR2008323} for a survey).

The present article is a further step in the program of stability up to a shift.

In this paper, we continue the ideas introduced in \cite{scalar_move_entire_solution}. In \cite{scalar_move_entire_solution}, it is shown that the generalized characteristics of $u$ can be used as shift functions to kill growth in $L^2$ between a piecewise smooth solution $\bar{u}$ and weak solution to \eqref{system} entropic for the entropy $\eta$. Further, using the generalized characteristic as a shift function provides various benefits over using the previous shift function constructions, as discussed in \cite{scalar_move_entire_solution}. 

In this paper, we bring novel ideas from the scalar case in \cite{scalar_move_entire_solution} to the systems case. In the systems case, we need to use the theory of a-contraction. 

For the case of scalar, the generalized characteristics for $u$ are the natural shift functions to be using. In the systems case, we use a shift function which again is based on the generalized characteristics, but with a correction where the shift travels at greater-than-characteristic-speed due to a-contraction and the existence of multiple shock families in the systems case. 

On top of the benefits for generalized-characteristic-based shifts mentioned in \cite{scalar_move_entire_solution} (such as simplicity of analysis, ease of construction, enhanced control on the shifts, and strictly negative entropy creation) the use of generalized-characteristic-based shifts for the \emph{systems case} allows for simplified proofs compared to the previous state-of-the-art a-contraction result, \cite{MR3519973}. By having very obvious control on the speed of  generalized-characteristic-based shifts, we are able to obviate the need for many of the computations in the foregoing analysis \cite{MR3519973}.

For systems of conservation laws in one space dimension such as \eqref{system} (including the scalar conservation laws), we have non-uniqueness for solutions. We impose entropy conditions such as  \eqref{entropy_condition_distributional_system_chitchat}, motivated by physics, to try to weed out ``nonphysical'' solutions which have physical entropy decreasing (or according to \eqref{entropy_condition_distributional_system_chitchat}, mathematical entropy increasing). Remark that requiring more than one entropy condition (for more than one entropy) is impractical -- many systems only admit a single nontrivial entropy. In the scalar case, this approach has had tremendous success. In fact, requiring solutions satisfy the entropy condition \eqref{entropy_condition_distributional_system_chitchat} for at least one strictly convex entropy in $C^1$ is enough to get uniqueness for solutions (see \cite{panov_uniquness,delellis_uniquneness,2017arXiv170905610K}). However, even for the scalar case proving uniqueness with a single entropy condition has proved difficult. The first result \cite{panov_uniquness} was not until 1994. Furthermore, the first two results \cite{panov_uniquness,delellis_uniquneness} use techniques limited to the scalar case. They use the special connection between scalar conservation laws in one space dimension and Hamilton--Jacobi equations: the space derivative of the solution to a Hamilton--Jacabi equation is formally the solution to the associated scalar conservation law. Notably, \cite{2017arXiv170905610K} gives a proof of the single entropy condition for scalar conservation laws which works directly on the conservation law and utilizes the theory of shifts.
Moreover, progress for uniqueness of entropic solutions to \emph{systems} of conservation laws has been slow. The best theory so far is the Bressan, Crasta, and Piccoli  $L^1$ theory \cite{MR1686652} for uniqueness in the class of solutions with small total variation.  It would be interesting however to study the uniqueness of these solutions amongst a larger class. For example, existence of solutions with large data is known for the $2\times2$ Euler system -- but the uniqueness theory for such solutions with large data lags behind.

The situation for the hyperbolic conservation laws in multiple space dimensions is even more dire --  there is non-uniqueness  for entropic  solutions to incompressible and compressible Euler by virtue of the many highly oscillatory solutions created via convex integration or related techniques. For incompressible Euler, see two papers by De Lellis and Sz\'ekelyhidi \cite{MR2600877,MR2564474}. For compressible Euler, see \cite{MR3352460,MR3269641,MR3744380}.

However, there is still the possibility of pushing forward the theory of \emph{uniqueness} for hyperbolic systems of conservation laws in one space dimension. The current paper is a step in that direction -- utilizing the $L^2$-type relative entropy method and the constantly evolving theory of shifts.

In this article, we use the method of relative entropy, the theory of shifts and a-contraction. These theories are not perturbative. They enable us to get results without small data limitations. Further, by the nature of these theories, we only use a single entropy condition.

We present our main and most important theorem regarding $L^2$-type stability and uniqueness results. The hypotheses $(\mathcal{H})$ and $(\mathcal{H})^*$ in the theorem depend only on the hyperbolic part of the system \eqref{system} and the fixed piecewise-smooth solution $\bar{u}$. The hypotheses are related to conditions on 1-shocks and n-shocks and in particular are satisfied by the isentropic Euler and full Euler systems. These hypotheses are explained in detail in \Cref{hypotheses_on_system}. 

\begin{theorem}[Main theorem -- $L^2$ stability for entropic piecewise-Lipschitz solutions to hyperbolic systems of balance laws]\label{local_stability_systems}

 Fix $R,T>0$.

Fix $i\in\{1,n\}$. Assume that $u,\bar{u}\in L^\infty(\mathbb{R}\times[0,T))$. If $\bar{u}$ contains a 1-shock, assume the hypotheses $(\mathcal{H})$ hold. Likewise, if $\bar{u}$ contains an n-shock, assume the hypotheses $(\mathcal{H})^*$ hold.  Assume that $u$ and $\bar{u}$ are entropic for the entropy $\eta\in C^3(\mathbb{R}^n)$. Assume that $\bar{u}$ is Lipschitz continuous on $\{(x,t)\in\mathbb{R}\times[0,T) | x<s(t)\}$ and on $\{(x,t)\in\mathbb{R}\times[0,T) | x>s(t)\}$, where $s:[0,T)\to\mathbb{R}$ is a Lipschitz function . Assume also that $u$ verifies the strong trace property (\Cref{strong_trace_definition}).
  Assume also that there exists $\rho>0$ such that for all $t\in[0,T)$ 
\begin{align}\label{gap_local_case}
\abs{\bar{u}(s(t)+,t)-\bar{u}(s(t)-,t)}>\rho.
\end{align}

Then there exists a Lipschitz continuous function $X:[0,T)\to\mathbb{R}$ with $X(0)=0$ and constants $\mu_1,\mu_2,r>0$ such that, 

\begin{align}\label{main_local_stability_result}
\int\limits_{-R+s(0)}^{R+s(0)}\abs{u(x,t_0)-\bar{u}(x+X(t_0),t_0)}^2\,dx\leq \mu_2 e^{\mu_1 t_0}\int\limits_{-R-rt_0+s(0)}^{R+rt_0+s(0)}\abs{u^0(x)-\bar{u}^0(x)}^2\,dx ,
\end{align}
for all $t_0\in[0,T)$.

Moreover, we have control on $X$:
\begin{align}\label{L2_control_shift_piecewise_systems}
\int\limits_0^{t_0} (\dot{X}(t))^2\,dt\leq \mu_2(1+e^{\mu_1 t_0})\int\limits_{-R-rt_0+s(0)}^{R+rt_0+s(0)}\abs{u^0(x)-\bar{u}^0(x)}^2\,dx.
\end{align} 

\end{theorem}
\begin{remark}
\hfill

\begin{itemize}
\item
The constants $\mu_1,\mu_2>0$ depend on $a$, $\rho$, $\norm{u}_{L^\infty}$, $\norm{\bar{u}}_{L^\infty}$, and bounds on the derivatives of $\eta$ on the range of $u$ and $\bar{u}$.    In addition, $\mu_1$ depends on $C_G$ (see \eqref{G_acts_like}), $\mbox{Lip}[\bar{u}]$, $R$, $T$, and bounds on the derivatives of $f$ on the range of $u$ and $\bar{u}$. Note that $r$ only depends on bounds on the derivatives of $f$ and $\eta$ (on the range of $u$ and $\bar{u}$). 
\item
As opposed to \eqref{G_acts_like_2}, the proof of \Cref{local_stability_systems} will in fact go through whenever we have an estimate of the form
\begin{align}\label{systems_new_estimate_remark_more_general}
\abs{\int\limits_{x_1}^{x_2} \nabla\eta(u(x,t)|\bar{u}(x+X(t),t))G(u(\cdot,t))(x)\,dx} \leq C \int\limits_{x_1}^{x_2}\abs{\nabla\eta(u(x,t)|\bar{u}(x+X(t),t))}\,dx,
\end{align}
for $x_1,x_2\in\mathbb{R}$ and some constant $C>0$. Note that $u\in L^\infty$ and \eqref{G_acts_like_2} implies \eqref{systems_new_estimate_remark_more_general}.
\item
Note that H\"older's inequality and \eqref{L2_control_shift_piecewise_systems} give control on the shift in the form of 
\begin{align}
\frac{1}{t_0}\int\limits_0^{t_0}\abs{\dot{X}(t)}\,dt \leq \frac{\sqrt{\mu_2(1+e^{\mu_1 t_0})}}{\sqrt{t_0}}\norm{u^0(\cdot)-\bar{u}^0(\cdot)}_{L^2(-R-rt_0+s(0),R+rt_0+s(0))}.
\end{align}
\item
Note that by Property (b) of $(\mathcal{H}1)$ or $(\mathcal{H}1)^*$, condition \eqref{gap_local_case} is equivalent to the existence of a $\tilde{\rho}>0$ such that for all $t\in[0,T)$ 
\begin{align}\label{gap_local_case}
r(t)>\tilde{\rho},
\end{align}
where $r(t)$ satisfies $S^i_{\bar{u}(s(t)-,t)}(r(t))=\bar{u}(s(t)+,t)$.
\end{itemize}
\end{remark}

\vspace{.07in}

The outline of the paper is as follows: in \Cref{hypotheses_on_system}, we give our hypotheses on the system. In \Cref{technical_lemmas}, we present technical lemmas. In \Cref{construction_of_the_shift}, we construct the shift with the additional entropy dissipation. Finally, in \Cref{proof_main_theorem} we prove the main theorem by using the  additional entropy dissipation from the shift to translate in $x$ the piecewise-smooth solution artificially.

\section{Hypotheses on the system}\label{hypotheses_on_system}
We will consider the following structural hypotheses $(\mathcal{H})$, $(\mathcal{H})^*$ on the system \eqref{system}, 

\eqref{entropy_condition_distributional_system_chitchat}  regarding the 1-shock and n-shock curves (they are closely related to hypotheses in \cite{Leger2011} and \cite{MR3519973}). For a fixed piecewise smooth solution $\bar{u}$ (as in the context of the main theorem \Cref{local_stability_systems}):
\hfill \break

\begin{itemize}
\item
$(\mathcal{H}1)$: (Family of 1-shocks verifying the Liu condition) There exists $r_0>0$ such that for all $u_L\in\{\bar{u}(s(t)-,t) | t\in[0,T)\}\coloneqq I_{-}$, and for all $u\in B_{r_0}(u_L)$, there is a 1-shock curve (issuing from $u$) $S_u^1\colon[0,s_u)\to \mathcal{V}$ (possibly $s_u=\infty$) parameterized by arc length. Moreover, $S_u^1(0)=u$ and the Rankine-Hugoniot jump condition holds:
\begin{align}
f(S_u^1(s))-f(u)=\sigma^1_u(s)(S_u^1(s)-u),
\end{align}
where $\sigma^1_u(s)$ is the velocity function. The map $u\mapsto s_u$ is Lipschitz on $\mathcal{V}$. Further, the maps $(s,u)\mapsto S_u^1(s)$ and $(s,u)\mapsto \sigma^1_u(s)$ are both $C^1$ on $\{(s,u)|s\in [0,s_u), u\in\mathcal{V}\}$, and the following conditions are satisfied: 
\begin{align*}
&\mbox{(a) (Liu entropy condition) } \frac{\mbox{d}}{\mbox{d}s} \sigma^1_u(s) <0,\hspace{.2in} \sigma^1_u(0)=\lambda_1(u), 
\\&\mbox{(b) (shock ``strengthens'' with $s$) }  \frac{\mbox{d}}{\mbox{d}s}\eta(u|S_u^1(s))>0, \hspace{.2in}\mbox{for all } s>0,
\\&\mbox{(c) (the shock curve cannot wrap tightly around itself)}
\\&\hspace{.5in}\mbox{For all $R>0$, there exists $\tilde{S}>0$ such that}  
\\
&\hspace{.8in}\Big\{S^1_{u}(s) \Big| s\in[0.s_u), \abs{u}\leq R \mbox{ and } \abs{S^1_u(s)}\leq R\Big\} \subseteq \Big\{S^1_u(s) \Big| \abs{u} \leq R \mbox{ and } s\leq \tilde{S}\Big\}.
\end{align*}

\item
$(\mathcal{H}2)$: If $(u_L,u_R)$ is an entropic Rankine-Hugoniot discontinuity with shock speed $\sigma$, then $\sigma> \lambda_1(u_R)$.

\item
$(\mathcal{H}3)$: If $(u_L,u_R)$ (with $u_L\in B_{r_0}(\tilde{u}_L)$, for $\tilde{u}_L\in I_{-}$) is an entropic Rankine-Hugoniot discontinuity with shock speed $\sigma$ verifying
\begin{align}
\sigma\leq\lambda_1(u_L),
\end{align}
then $u_R$ is in the image of $S_{u_L}^1$. In other words, there exists $s_{u_R}\in[0,s_{u_L})$ such that $S_{u_L}^1(s_{u_R})=u_R$ (and by implication, $\sigma=\sigma^1_{u_L}(s_{u_R})$).
\end{itemize}

Similarly, we will consider the following structural hypotheses $(\mathcal{H})^*$ on the system \eqref{system}, \eqref{entropy_condition_distributional_system_chitchat} regarding the n-shock curves:
\hfill \break

\begin{itemize}
\item
$(\mathcal{H}1)^*$: (Family of n-shocks verifying the Liu condition) There exists $r_0>0$ such that for all $u_R\in\{\bar{u}(s(t)+,t) | t\in[0,T)\}\coloneqq I_{+}$, and for all $u\in B_{r_0}(u_R)$, there is an n-shock curve (issuing from $u$) $S_u^n\colon[0,s_u)\to \mathcal{V}$ (possibly $s_u=\infty$) parameterized by arc length. Moreover, $S_u^n(0)=u$ and the Rankine-Hugoniot jump condition holds:
\begin{align}
f(S_u^n(s))-f(u)=\sigma^n_u(s)(S_u^n(s)-u),
\end{align}
where $\sigma^n_u(s)$ is the velocity function. The map $u\mapsto s_u$ is Lipschitz on $\mathcal{V}$. Further, the maps $(s,u)\mapsto S_u^n(s)$ and $(s,u)\mapsto \sigma^n_u(s)$ are both $C^1$ on $\{(s,u)|s\in [0,s_u), u\in\mathcal{V}\}$, and the following conditions are satisfied: 
\begin{align*}
&\mbox{(a) (Liu entropy condition) } \frac{\mbox{d}}{\mbox{d}s} \sigma^n_u(s) >0,\hspace{.2in} \sigma^n_u(0)=\lambda_n(u), 
\\&\mbox{(b) (shock ``strengthens'' with $s$) }  \frac{\mbox{d}}{\mbox{d}s}\eta(u|S_u^n(s))>0, \hspace{.2in}\mbox{for all } s>0,
\\&\mbox{(c) (the shock curve cannot wrap tightly around itself)}
\\&\hspace{.5in}\mbox{For all $R>0$, there exists $\tilde{S}>0$ such that}  
\\
&\hspace{.8in}\Big\{S^n_{u}(s) \Big| s\in[0.s_u), \abs{u}\leq R \mbox{ and } \abs{S^n_u(s)}\leq R\Big\} \subseteq \Big\{S^n_u(s) \Big| \abs{u} \leq R \mbox{ and } s\leq \tilde{S}\Big\}.
\end{align*}

\item
$(\mathcal{H}2)^*$: If $(u_R,u_L)$ is an entropic Rankine-Hugoniot discontinuity with shock speed $\sigma$, then $\sigma< \lambda_n(u_L)$.

\item
$(\mathcal{H}3)^*$: If $(u_R,u_L)$ (with $u_R\in B_{r_0}(\tilde{u}_R)$, for $\tilde{u}_R\in I_{+}$) is an entropic Rankine-Hugoniot discontinuity with shock speed $\sigma$ verifying
\begin{align}
\sigma\geq\lambda_n(u_R),
\end{align}
then $u_L$ is in the image of $S_{u_R}^n$. In other words, there exists $s_{u_L}\in[0,s_{u_R})$ such that $S_{u_R}^n(s_{u_L})=u_L$ (and by implication, $\sigma=\sigma^n_{u_R}(s_{u_L})$).
\end{itemize}

\begin{remark}
See \cite{Leger2011,MR3519973} for remarks on these hypotheses. We include them here for completeness. In particular, 
\hfill \break
\begin{itemize}
\item
Note that the system \eqref{system} verifies the hypotheses $(\mathcal{H}1)$-$(\mathcal{H}3)$ on the 1-shock family if and only if the system
\begin{align}
\begin{cases}
\partial_t u - \partial_x f(u)=G(u(\cdot,t))(x),\mbox{ for } x\in\mathbb{R},\mbox{ } t>0,\\
u(x,0)=u^0(x) \mbox{ for } x\in\mathbb{R}.
\end{cases}
\end{align}
verifies the properties $(\mathcal{H}1)^*$-$(\mathcal{H}3)^*$ for the n-shock family. It is in this way that $(\mathcal{H}1)$-$(\mathcal{H}3)$ are dual to $(\mathcal{H}1)^*$-$(\mathcal{H}3)^*$.
\item
On top of the Liu entropy condition (Property (a) in $(\mathcal{H}1)$), we also assume Property (b), which says that the 1-shock strength grows along the 1-shock curve $S^1_{u_L}$ when measured via the pseudo-distance of the relative entropy (recall that the map $(u,v)\mapsto\eta(u|v)$ measures  $L^2$-distance somehow -- see \eqref{entropy_relative_L2_control_system}). This growth condition arises naturally in the study of admissibility criteria for systems of conservation laws. In particular, Property (b) ensures that Liu admissible shocks are entropic for the entropy $\eta$  even for moderate-to-strong shocks (see \cite{MR1600904,MR0093653,MR2053765}). 

In \cite{MR3338447}, Barker, Freist\"{u}hler, and Zumbrun show that stability  and in particular contraction fails to hold for the full Euler system if we replace Property (b) with
\begin{align}
\frac{\mbox{d}}{\mbox{d}s}\eta(S_u^1(s))>0,\hspace{.2in} s>0.
\end{align}
This shows that it is better to measure shock strength using the relative entropy rather than the entropy itself.
\item
Recall the famous Lax E-condition for an  i-shock $(u_L,u_R,\sigma)$,
\begin{align}
\lambda_i(u_R)\leq\sigma\leq\lambda_i(u_L).
\end{align}
The hypothesis $(\mathcal{H}2)$ is implied by the first half of the Lax E-condition along with the hyperbolicity of the system \eqref{system}. In addition, we do not allow for right 1-contact discontinuities. 
\item
The hypothesis $(\mathcal{H}3)$ is a statement about the well-separation of the 1-shocks from all other Rankine-Hugoniot discontinuities entropic for $\eta$; the 1-shocks do not interfere with any other shocks. In particular, $(\mathcal{H}3)$ will hold for any strictly hyperbolic system in the form \eqref{system} if all Rankine-Hugoniot discontinuities $(u_L,u_R,\sigma)$ entropic for $\eta$ lie on an i-shock curve for some $i$ and the extended Lax admissibility condition holds:
\begin{align}\label{extended_lax_admissibility_condition}
\lambda_{i-1}(u_L) \leq \sigma \leq \lambda_{i+1} (u_R),
\end{align}
where $\lambda_0\coloneqq -\infty$ and $\lambda_{n+1}\coloneqq\infty$. Moreover, we only use the first inequality in \eqref{extended_lax_admissibility_condition} and the fact that $\lambda_1(u)\leq \lambda_{i-1}(u)$ for all $u\in\mathcal{V}$ and for all $i>1$.

Furthermore, note that for \emph{any} strictly hyperbolic system in the form \eqref{system}, if  $u_R$ and $u_L$ live in a fixed compact set, then there exists $\delta>0$ such that \eqref{extended_lax_admissibility_condition} will hold if $\abs{u_R-u_L}\leq\delta$. Similarly,  for any strictly hyperbolic system endowed with a strictly convex entropy, all Rankine-Hugoniot discontinuities $(u_L,u_R,\sigma)$ entropic for $\eta$ will locally be in the form $S^i_{u_L}(s)=u_R$ for some $s>0$, and where $S^i_{u_L}$ is the i-shock curve issuing  from $u_L$. See \cite[Theorem 1.1, p.~140]{lefloch_book} and more generally \cite[p.~140-6]{lefloch_book}. For the full Euler system , $(\mathcal{H}3)$ will hold regardless of the size of the shock $(u_L,u_R)$.
\item
Fix $B,\rho>0$. Then, for all $u\in\mathcal{V}$ with $\abs{u}\leq B$ and for all $s\in[\rho,B]$, we have
\begin{equation}
\begin{aligned}\label{s_shock_strength_comparable_part1}
(s-\rho)\inf_{{\substack{u\in\mathcal{V},\hspace{.02in}\abs{u}\leq B\\t\in[\rho,B]}}}\frac{\mbox{d}}{\mbox{d}t}\eta(u|S_u^1(t))&\leq\eta(u|S_u^1(s))\\
&=\int\limits_\rho^s \frac{\mbox{d}}{\mbox{d}t}\eta(u|S_u^1(t))\,\mbox{d}t \leq (s-\rho)\sup_{{\substack{u\in\mathcal{V},\hspace{.02in}\abs{u}\leq B\\t\in[\rho,B]}}}\frac{\mbox{d}}{\mbox{d}t}\eta(u|S_u^1(t)).
\end{aligned}
\end{equation}
Note that 
\begin{align}
0<\inf_{{\substack{u\in\mathcal{V},\hspace{.02in}\abs{u}\leq B\\t\in[\rho,B]}}}\frac{\mbox{d}}{\mbox{d}t}\eta(u|S_u^1(t))
\end{align}
and
\begin{align}
0<\sup_{{\substack{u\in\mathcal{V},\hspace{.02in}\abs{u}\leq B\\t\in[\rho,B]}}}\frac{\mbox{d}}{\mbox{d}t}\eta(u|S_u^1(t))<\infty
\end{align}
due to $(\mathcal{H}1)$.

Recall also that by hypothesis $(\mathcal{H}1)$, $S_u^1$ is parameterized by arc length. Thus, $\abs{S_u^1(s)-u}\leq B$ for all $s\in[0,B]$. We can then use \eqref{s_shock_strength_comparable_part1} and \Cref{entropy_relative_L2_control_system} to get,
\begin{align}\label{shock_strength_comparable_s_systems1}
(s-\rho)d_1\leq\abs{u-S_u^1(s)}^2\leq (s-\rho) d_2
\end{align}
for all $u\in\mathcal{V}$ with $\abs{u}\leq B$ and for all $s\in[\rho,B]$. The constants $d_1,d_2>0$ depend only on $B$ and $\rho$. This says that  $s-\rho$ is comparable to the  shock strength $\abs{u-S_u^1(s)}^2$.

\item
On the state space $\mathcal{V}$ where the strictly convex entropy $\eta$ is defined, the system \eqref{system} is hyperbolic. Further, by virtue of $f\in C^2(\mathcal{V})$, the eigenvalues of $\nabla f (u)$ vary continuously on the state space $\mathcal{V}$. Further, if the eigenvalue $\lambda_1(u)$ ($\lambda_n(u)$) is simple for $u\in\mathcal{V}$ (such as when the system \eqref{system} is strictly hyperbolic), the map $u\mapsto \lambda_1(u)$ ($u\mapsto \lambda_n(u)$) will be in $C^1(\mathcal{V})$ due to the implicit function theorem.
\end{itemize}
\end{remark}

We study solutions $u$ to \eqref{system} among the class of functions verifying a strong trace property (first introduced in \cite{Leger2011}):

\begin{definition}\label{strong_trace_definition}
Fix $T>0$. Let $u\colon\mathbb{R}\times[0,T)\to\mathbb{R}^n$ verify $u\in L^\infty(\mathbb{R}\times[0,T))$. We say $u$ has the \emph{strong trace property} if for every fixed Lipschitz continuous map $h\colon [0,T)\to\mathbb{R}$, there exists $u_+,u_-\colon[0,T)\to\mathbb{R}^n$ such that
\begin{align}
\lim_{n\to\infty}\int\limits_0^{t_0}\esssup_{y\in(0,\frac{1}{n})}\abs{u(h(t)+y,t)-u_+(t)}\,dt=\lim_{n\to\infty}\int\limits_0^{t_0}\esssup_{y\in(-\frac{1}{n},0)}\abs{u(h(t)+y,t)-u_-(t)}\,dt=0
\end{align}
for all $t_0\in(0,T)$.
\end{definition}

Note that for example a function $u\in L^\infty(\mathbb{R}\times[0,T))$ will satisfy the strong trace property if for each fixed $h$, the right and left limits
\begin{align}
\lim_{y\to0^{+}}u(h(t)+y,t)\hspace{.7in}\mbox{and}\hspace{.7in}\lim_{y\to0^{-}}u(h(t)+y,t)
\end{align}
exist for almost every $t$. In particular, a function $u\in L^\infty(\mathbb{R}\times[0,T))$ will have strong traces according to \Cref{strong_trace_definition} if $u$ has a representative which is in $BV_{\text{loc}}$. However, the strong trace property is weaker than $BV_{\text{loc}}$.

\section{Technical Lemmas}\label{technical_lemmas}
Throughout this paper, we use the following definition for the relative flux
\begin{align}\label{Z_def}
f(a|b)\coloneqq f(a)-f(b)-\nabla f (b)(a-b),
\end{align}
and the relative $\nabla\eta$: for $a,b\in\mathbb{M}^{n\times 1}$,
\begin{align}
\nabla\eta(a|b)\coloneqq \nabla\eta(a)-\nabla\eta(b)-[a-b]^T\nabla^2\eta(b).
\end{align}.

The following lemma from \cite{MR3537479} describes how the relative entropy obeys a sort of triangle inequality:
\begin{lemma}[Structural lemma from \cite{MR3537479} - triangle inequality for the relative entropy]\label{triangle_inequality_systems_entropy}
For any $u,v,w\in\mathcal{V}$, we have
\begin{align}
\eta(u|w)+\eta(w|v)=\eta(u|v)+(\nabla\eta(w)-\nabla\eta(v))\cdot(w-u),
\end{align}
and
\begin{align}
q(u;w)+q(w;v)=q(u;v)+(\nabla\eta(w)-\nabla\eta(v))\cdot(f(w)-f(u)).
\end{align}

Thus, for any $\sigma\in\mathbb{R}$,
\begin{equation}
\begin{aligned}\label{48_kang_vasseur_a_contraction}
q(u;v)-\sigma\eta(u|v)=&(q(u;w)-\sigma\eta(u;w))+(q(w;v)-\sigma\eta(w|v))
\\
&-(\nabla\eta(w)-\nabla\eta(v))\cdot(f(w)-f(u)-\sigma(w-u)).
\end{aligned}
\end{equation}
\end{lemma}
The proof of \Cref{triangle_inequality_systems_entropy} follows immediately from the definition of $q(\hspace{.015in}\cdot\hspace{.015in};\hspace{.015in}\cdot\hspace{.015in})$ and $\eta(\hspace{.015in}\cdot\hspace{.015in}|\hspace{.015in}\cdot\hspace{.015in})$. In particular, see \cite[p.~360-1]{MR3519973} for a simple proof.

\begin{lemma}\label{a_cond_lemma_itself}
Fix $B>0$. Then there exists a constant $C>0$ depending on $B$ such that the following holds:
 
If $u_L,u_R\in\mathcal{V}$ with $\abs{u_L},\abs{u_R}\leq B$, then whenever $\alpha,\theta\in(0,1)$ verify 
\begin{align}
\label{cond_a}
\alpha<\frac{\theta^2}{C},
\end{align}
then $R_a\coloneqq\{u | \eta(u|u_L)\leq a\eta(u|u_R)\}\subset B_{\theta}(u_L)$ for all $0<a<\alpha$. 
\end{lemma}
\begin{remark}
The set $R_a$ is compact.
\end{remark}
The proof of \Cref{a_cond_lemma_itself} is found in the proof of Lemma 4.3 in \cite{MR3519973}. We repeat the proof in \Cref{appendix_a_cond_lemma_itself} for the reader's convenience.

\hfill

The following Lemma gives the entropy dissipation caused by changing the domain of integration, translating the solution $\bar{u}$ in $x$ (by a function $X(t)$), and from the source term $G$.

\begin{lemma}[Local entropy dissipation rate]\label{local_entropy_dissipation_rate_systems}
Fix $T>0$. Let $u,\bar{u}\in L^\infty(\mathbb{R}\times[0,T))$ be weak solutions to \eqref{system}.  Assume that $u$ and $\bar{u}$ are entropic for the entropy $\eta$. Assume that $\bar{u}$ is Lipschitz continuous on $\{(x,t)\in\mathbb{R}\times[0,T) | x<s(t)\}$ and on $\{(x,t)\in\mathbb{R}\times[0,T) | x>s(t)\}$, where $s:[0,T)\to\mathbb{R}$ is a Lipschitz function . Assume also that $u$ verifies the strong trace property (\Cref{strong_trace_definition}).  Let $h_1,h_2, X:[0,T)\to\mathbb{R}$ be Lipschitz continuous functions with the property that there exists $\delta>0$ such that $h_2(t)-h_1(t)\geq \delta$ for all $t\in[0,T)$. Assume also that for all $t\in[0,T)$, $s(t)-X(t)$ is not in the open set $(h_1(t),h_2(t))$.

Then,
\begin{equation}
\begin{aligned}\label{local_compatible_dissipation_calc}
&\int\limits_{0}^{t_0} \bigg[q(u(h_1(t)+,t);\bar{u}((h_1(t)+X(t))+,t))-q(u(h_2(t)-,t);\bar{u}((h_2(t)+X(t))-,t))
\\
&\hspace{.7in}+\dot{h}_2(t)\eta(u(h_2(t)-,t)|\bar{u}((h_2(t)+X(t))-,t))
\\
&\hspace{.7in}-\dot{h}_1(t)\eta(u(h_1(t)+,t)|\bar{u}((h_1(t)+X(t))+,t))\bigg]\,dt
\\
&\hspace{.5in}\geq
\int\limits_{h_1(t_0)}^{h_2(t_0)}\eta(u(x,t_0)|\bar{u}(x+X(t_0),t_0))\,dx
-\int\limits_{h_1(0)}^{h_2(0)}\eta(u^0(x)|\bar{u}^0(x))\,dx
\\
&\hspace{.7in}+\int\limits_{0}^{t_0}\int\limits_{h_1(t)}^{h_2(t)}\Bigg(\partial_x \bigg|_{(x+X(t),t)}\hspace{-.45in} \nabla\eta(\bar{u}(x,t))\Bigg)f(u(x,t)|\bar{u}(x+X(t),t))
\\
&\hspace{.7in}+\Bigg(2\partial_x\bigg|_{(x+X(t),t)}\hspace{-.45in}\bar{u}^T(x,t)\dot{X}(t)\Bigg)\nabla^2\eta(\bar{u}(x+X(t),t))[u(x,t)-\bar{u}(x+X(t),t)]
\\
&\hspace{.7in}-\nabla\eta(u(x,t)|\bar{u}(x+X(t),t))G(u(\cdot,t))(x)
\\
&+
\Bigg(G(\bar{u}(\cdot,t))(x+X(t))-G(u(\cdot,t))(x)\Bigg)^T\nabla^2\eta(\bar{u}(x+X(t),t))[u(x,t)-\bar{u}(x+X(t),t)]\,dxdt.
\end{aligned}
\end{equation}
\end{lemma}
\begin{proof}
This proof is based on a similar argument in \cite{scalar_move_entire_solution}.

\hfill\break
\uline{Step 1}
\hfill\break
We show that for all positive, Lipschitz continuous test functions $\phi:\mathbb{R}\times[0,T)\to\mathbb{R}$ with compact support and that vanish on the set $\{(x,t)\in\mathbb{R}\times[0,T) | x=s(t)-X(t)\}$, we have
\begin{equation}
\begin{aligned}\label{combined1}
\int\limits_{0}^{T} \int\limits_{-\infty}^{\infty} [\partial_t \phi \eta(u(x,t)|\bar{u}(x+X(t),t))+\partial_x \phi q(u(x,t);\bar{u}(x+X(t),t))]\,dxdt\hspace{1.5in} \\
+  \int\limits_{-\infty}^{\infty}\phi(x,0)\eta(u^0(x)|\bar{u}^0(x))\,dx\hspace{3in} \\
\geq
\int\limits_{0}^{T} \int\limits_{-\infty}^{\infty}\phi\Bigg[\Bigg(
\partial_x \bigg|_{(x+X(t),t)} \hspace{-.45in}\nabla\eta(\bar{u}(x,t))\Bigg)f(u(x,t)|\bar{u}(x+X(t),t))
\\
+\Bigg(2\partial_x\bigg|_{(x+X(t),t)}\hspace{-.45in}\bar{u}^T(x,t)\dot{X}(t)\Bigg)\nabla^2\eta(\bar{u}(x+X(t),t))[u(x,t)-\bar{u}(x+X(t),t)]
\\
-\nabla\eta(u(x,t)|\bar{u}(x+X(t),t))G(u(\cdot,t))(x)
\\
+
\Bigg(G(\bar{u}(\cdot,t))(x+X(t))-G(u(\cdot,t))(x)\Bigg)^T\nabla^2\eta(\bar{u}(x+X(t),t))[u(x,t)-\bar{u}(x+X(t),t)]
\Bigg]\,dxdt.
\end{aligned}
\end{equation}

Note that \eqref{combined1} is the analogue in our case of the key estimate used in Dafermos's proof of weak-strong stability, which gives a relative version of the entropy inequality (see equation (5.2.10) in \cite[p.~122-5]{dafermos_big_book}). The proof of \Cref{combined1} is based on the famous weak-strong stability proof of Dafermos and DiPerna \cite[p.~122-5]{dafermos_big_book}. To take into account the entropy production due to translating the solution $\bar{u}$ by the function $X$, we use the argument introduced in \cite{scalar_move_entire_solution}.

%

Note that on the complement of the set $\{(x,t)\in\mathbb{R}\times[0,T) | x=s(t)\}$, $\bar{u}$ is smooth  and so we have the exact equalities,
\begin{align}
\partial_t\bigg|_{(x,t)}\hspace{-.21in}\big(\bar{u}(x,t)\big)+\partial_x\bigg|_{(x,t)}\hspace{-.21in}\big(f(\bar{u}(x,t))\big)&=G(\bar{u}(\cdot,t))(x),\label{solves_equation}\\
\partial_t\bigg|_{(x,t)}\hspace{-.21in}\big(\eta(\bar{u}(x,t))\big)+\partial_x\bigg|_{(x,t)}\hspace{-.21in}\big(q(\bar{u}(x,t))\big)&=\nabla\eta(\bar{u}(x,t))G(\bar{u}(\cdot,t))(x).\label{solves_entropy}
\end{align}

Thus for any Lipschitz continuous function $X: [0,T)\to\mathbb{R}$ with $X(0)=0$  we have on the complement of the set $\{(x,t)\in\mathbb{R}\times[0,T) | x=s(t)-X(t)\}$, 
\begin{equation}
\begin{aligned}\label{solves_equation_shift}
\partial_t\bigg|_{(x,t)}\hspace{-.21in}&\big(\bar{u}(x+X(t),t)\big)+\partial_x\bigg|_{(x,t)}\hspace{-.21in}\big(f(\bar{u}(x+X(t),t))\big)=
\\
&\hspace{1.5in}\Bigg(\partial_x\bigg|_{(x+X(t),t)}\hspace{-.45in}\big(\bar{u}(x,t)\big)\Bigg)\dot{X}(t)+G(\bar{u}(\cdot,t))(x+X(t)),
\end{aligned}
\end{equation}
and
\begin{equation}
\begin{aligned}\label{solves_entropy_shift}
\partial_t\bigg|_{(x,t)}\hspace{-.21in}&\big(\eta(\bar{u}(x+X(t),t))\big)+\partial_x\bigg|_{(x,t)}\hspace{-.21in}\big(q(\bar{u}(x+X(t),t))\big)=
\\
&\nabla\eta(\bar{u}(x+X(t),t))\Bigg(\partial_x\bigg|_{(x+X(t),t)}\hspace{-.45in}\big(\bar{u}(x,t)\big)\Bigg)\dot{X}(t)+\nabla\eta(\bar{u}(x+X(t),t))G(\bar{u}(\cdot,t))(x+X(t)).
\end{aligned}
\end{equation}

We can now imitate the weak-strong stability proof in \cite[p.~122-5]{dafermos_big_book}, using \eqref{solves_equation_shift} and \eqref{solves_entropy_shift} instead of \eqref{solves_equation} and \eqref{solves_entropy}.

Recall \eqref{Z_def}, which says

\begin{align}
f(u|\bar{u})\coloneqq f(u)-f(\bar{u})-\nabla f (\bar{u})(u-\bar{u}).
\end{align}
Remark that $f(u|\bar{u})$ is locally quadratic in $u-\bar{u}$. 

Fix any positive, Lipschitz continuous test function $\phi:\mathbb{R}\times[0,T)\to\mathbb{R}$ with compact support. Assume also that $\phi$ vanishes on the set $\{(x,t)\in\mathbb{R}\times[0,T) | x=s(t)-X(t)\}$. Then, we use that $u$ satisfies the entropy inequality in a distributional sense:
 \begin{equation}
\begin{aligned}\label{u_entropy_integral_formulation}
\int\limits_{0}^{T} \int\limits_{-\infty}^{\infty}\Bigg[\partial_t\phi\big(\eta(u(x,t))\big)+&\partial_x \phi \big(q(u(x,t))\big)\Bigg]\,dxdt+ \int\limits_{-\infty}^{\infty}\phi(x,0)\eta(u^0(x))\,dx
\\
&\geq-\int\limits_{0}^{T} \int\limits_{-\infty}^{\infty}\phi\nabla\eta(u(x,t))G(u(\cdot,t))(x)\,dxdt.
\end{aligned}
\end{equation}

 We also view \eqref{solves_entropy_shift} as a distributional equality:
 \begin{equation}
\begin{aligned}\label{solves_entropy_shift_integral_formulation}
\int\limits_{0}^{T} \int\limits_{-\infty}^{\infty}\Bigg[\partial_t\phi\big(\eta(\bar{u}(x+&X(t),t))\big)+\partial_x \phi \big(q(\bar{u}(x+X(t),t))\big)\Bigg]\,dxdt+ \int\limits_{-\infty}^{\infty}\phi(x,0)\eta(\bar{u}^0(x))\,dx
\\
&=-\int\limits_{0}^{T} \int\limits_{-\infty}^{\infty}\phi\Bigg[\nabla\eta(\bar{u}(x+X(t),t))\Bigg(\partial_x\bigg|_{(x+X(t),t)}\hspace{-.45in}\big(\bar{u}(x,t)\big)\Bigg)\dot{X}(t)
\\
&\hspace{.5in}+\nabla\eta(\bar{u}(x+X(t),t))G(\bar{u}(\cdot,t))(x+X(t))\Bigg]\,dxdt.
\end{aligned}
\end{equation}

To get \eqref{solves_entropy_shift_integral_formulation}, we do integration by parts twice on the right hand side of \eqref{solves_entropy_shift}. Once on the domain $\{(x,t)\in\mathbb{R}\times[0,T) | x<s(t)-X(t)\}$ and once on the domain $\{(x,t)\in\mathbb{R}\times[0,T) | x>s(t)-X(t)\}$. We don't have a boundary term along the set $\{(x,t)\in\mathbb{R}\times[0,T) | x=s(t)-X(t)\}$ because $\phi$ vanishes on this set.

 We subtract \eqref{solves_entropy_shift_integral_formulation} from \eqref{u_entropy_integral_formulation}, to get

\begin{equation}
\begin{aligned}\label{difference_entropy_equations}
\int\limits_{0}^{T} \int\limits_{-\infty}^{\infty} [\partial_t \phi \eta(u(x,t)|\bar{u}(x+X(t),t))+\partial_x \phi q(u(x,t),\bar{u}(x+X(t),t))]\,dxdt \hspace{.5in}\\
+  \int\limits_{-\infty}^{\infty}\phi(x,0)\eta(u^0(x)|\bar{u}^0(x))\,dx \hspace{2.5in}\\
\geq -\int\limits_{0}^{T} \int\limits_{-\infty}^{\infty} \Big(\partial_t \phi\nabla\eta(\bar{u}(x+X(t),t))[u(x,t)-\bar{u}(x+X(t),t)]
\\
+\partial_x \phi \nabla\eta(\bar{u}(x+X(t),t))[f(u(x,t))-f(\bar{u}(x+X(t),t))]\Big)\,dxdt 
\\
- \int\limits_{-\infty}^{\infty}\phi(x,0)\nabla\eta(\bar{u}^0(x))[u^0(x)-\bar{u}^0(x)]\,dx
\\
+
\int\limits_{0}^{T} \int\limits_{-\infty}^{\infty}\phi\Bigg[\nabla\eta(\bar{u}(x+X(t),t))\Bigg(\partial_x\bigg|_{(x+X(t),t)}\hspace{-.45in}\big(\bar{u}(x,t)\big)\Bigg)\dot{X}(t)
\\
+\nabla\eta(\bar{u}(x+X(t),t))G(\bar{u}(\cdot,t))(x+X(t))-\nabla\eta(u(x,t))G(u(\cdot,t))(x)\Bigg]\,dxdt.
\end{aligned}
\end{equation}

The function $u$ is a distributional solution to the system of conservation laws. Thus, for every Lipschitz continuous test function $\Phi:\mathbb{R}\times[0,T)\to \mathbb{M}^{1\times n}$ with compact support,
\begin{equation}
\begin{aligned}\label{u_solves_equation_integral_formulation}
\int\limits_{0}^{T} \int\limits_{-\infty}^{\infty} \Bigg[\partial_t\Phi u + \partial_x\Phi f(u) \Bigg]\,dxdt +\int\limits_{-\infty}^{\infty} \Phi(x,0)u^0(x)\,dx
\\
=-\int\limits_{0}^{T} \int\limits_{-\infty}^{\infty}\Phi G(u(\cdot,t))(x)\,dxdt.
\end{aligned}
\end{equation}

We also can rewrite \eqref{solves_equation_shift} in a distributional way, for $\Phi$ which have the additional property of vanishing on $\{(x,t)\in\mathbb{R}\times[0,T) | x=s(t)-X(t)\}$:
\begin{equation}
\begin{aligned}\label{shift_solves_equation_integral_formulation}
\int\limits_{0}^{T} \int\limits_{-\infty}^{\infty} \Bigg[\partial_t\Phi \bar{u}(x+X(t),t) + \partial_x\Phi f(\bar{u}(x+X(t),t)) \Bigg]\,dxdt +\int\limits_{-\infty}^{\infty} \Phi(x,0)\bar{u}^0(x)\,dx
\\
=-\int\limits_{0}^{T} \int\limits_{-\infty}^{\infty}\Phi \Bigg[\Bigg(\partial_x\bigg|_{(x+X(t),t)}\hspace{-.45in}\big(\bar{u}(x,t)\big)\Bigg)\dot{X}(t)+G(\bar{u}(\cdot,t))(x+X(t))\Bigg]\,dxdt.
\end{aligned}
\end{equation}
To prove \eqref{shift_solves_equation_integral_formulation}, on the right hand side of \eqref{solves_equation_shift} we again do integration by parts twice. Once on the domain $\{(x,t)\in\mathbb{R}\times[0,T) | x<s(t)-X(t)\}$ and once on the domain $\{(x,t)\in\mathbb{R}\times[0,T) | x>s(t)-X(t)\}$. We lose the boundary terms along $\{(x,t)\in\mathbb{R}\times[0,T) | x=s(t)-X(t)\}$ because $\Phi$ vanishes there. 

Then, we can choose 
\begin{align}\label{our_choice_for_Phi}
\phi\nabla\eta(\bar{u}(x+X(t),t))
\end{align} 
 as the test function $\Phi$, and subtract \eqref{shift_solves_equation_integral_formulation} from \eqref{u_solves_equation_integral_formulation}. We can extend the function \eqref{our_choice_for_Phi} to the set $\{(x,t)\in\mathbb{R}\times[0,T) | x=s(t)-X(t)\}$ by defining it to be zero. This extension is still Lipschitz continuous.

This yields,
\begin{equation}
\begin{aligned}\label{difference_of_solutions_equation}
\int\limits_{0}^{T} \int\limits_{-\infty}^{\infty} \Bigg[\partial_t[\phi\nabla\eta(\bar{u}(x+X(t),t))][u(x,t)- \bar{u}(x+X(t),t)] \hspace{2in} 
\\
+\partial_x[\phi\nabla\eta(\bar{u}(x+X(t),t))][f(u(x,t))-f(\bar{u}(x+X(t),t))] \Bigg]\,dxdt \\
+\int\limits_{-\infty}^{\infty} \phi(x,0)\nabla\eta(\bar{u}^0(x))[u^0(x)-\bar{u}^0(x)]\,dx\hspace{1.42in}
\\
=\int\limits_{0}^{T} \int\limits_{-\infty}^{\infty}\phi\nabla\eta(\bar{u}(x+X(t),t)) \Bigg[\Bigg(\partial_x\bigg|_{(x+X(t),t)}\hspace{-.45in}\big(\bar{u}(x,t)\big)\Bigg)\dot{X}(t)\\
+G(\bar{u}(\cdot,t))(x+X(t))-G(u(\cdot,t))(x)\Bigg]\,dxdt.
\end{aligned}
\end{equation}

Recall $\bar{u}$ is a classical solution on the complement of the set $\{(x,t)\in\mathbb{R}\times[0,T) | x=s(t)\}$ and verifies \eqref{solves_equation_shift}. Thus, on the complement of the set $\{(x,t)\in\mathbb{R}\times[0,T) | x=s(t)-X(t)\}$,
\begin{equation}
\begin{aligned}\label{notice_this}
&\partial_t\bigg|_{(x,t)}\hspace{-.21in}\big(\nabla\eta(\bar{u}(x+X(t),t))\big)=\Bigg(\partial_x\bigg|_{(x+X(t),t)}\hspace{-.45in}\bar{u}^T(x,t)\dot{X}(t)+\partial_t\bigg|_{(x+X(t),t)}\hspace{-.45in}\bar{u}^T(x,t)\Bigg)\nabla^2\eta(\bar{u}(x+X(t),t))
\\
&\hspace{1in}=
\Bigg(2\partial_x\bigg|_{(x+X(t),t)}\hspace{-.45in}\bar{u}^T(x,t)\dot{X}(t)
-\partial_x \bigg|_{(x+X(t),t)}\hspace{-.45in} \bar{u}^T(x,t)\big[\nabla f(\bar{u}(x+X(t),t))\big]^T\\
&\hspace{1.5in}+G^T(\bar{u}(\cdot,t))(x+X(t))\Bigg)\nabla^2\eta(\bar{u}(x+X(t),t))
\\
&\hspace{1in}=
\Bigg(2\partial_x\bigg|_{(x+X(t),t)}\hspace{-.45in}\bar{u}^T(x,t)\dot{X}(t)
+G^T(\bar{u}(\cdot,t))(x+X(t))\Bigg)\nabla^2\eta(\bar{u}(x+X(t),t))
\\
&\hspace{1.5in}-\partial_x \bigg|_{(x+X(t),t)}\hspace{-.45in} \bar{u}^T(x,t)\nabla^2\eta(\bar{u}(x+X(t),t))\nabla f(\bar{u}(x+X(t),t)),
\end{aligned}
\end{equation}
because $\big[\nabla f(\bar{u})\big]^T\nabla^2\eta(\bar{u})=\nabla^2\eta(\bar{u})\nabla f(\bar{u})$.

Thus, by \eqref{notice_this} and the definition of the relative flux in \eqref{Z_def},

\begin{equation}
\begin{aligned}\label{5.2.9}
\partial_t\bigg|_{(x,t)}\hspace{-.21in}\big(\nabla\eta(\bar{u}(x+X(t),t))\big)[u(x,t)-\bar{u}(x+X(t),t)]\hspace{2.5in}
\\
+\partial_x\bigg|_{(x,t)}\hspace{-.21in}\big(\nabla\eta(\bar{u}(x+X(t),t))\big)[f(u(x,t))-f(\bar{u}(x+X(t),t))]\hspace{1in}
\\
=
\partial_x \bigg|_{(x+X(t),t)}\hspace{-.45in} \bar{u}^T(x,t)\nabla^2\eta(\bar{u}(x+X(t),t)) f(u(x,t)|\bar{u}(x+X(t),t))
\\
+\Bigg(2\partial_x\bigg|_{(x+X(t),t)}\hspace{-.45in}\bar{u}^T(x,t)\dot{X}(t)
+G^T(\bar{u}(\cdot,t))(x+X(t))\Bigg)\nabla^2\eta(\bar{u}(x+X(t),t))[u(x,t)-\bar{u}(x+X(t),t)].
\end{aligned}
\end{equation}

We combine \eqref{difference_entropy_equations}, \eqref{difference_of_solutions_equation}, and \eqref{5.2.9} to get
\begin{equation}
\begin{aligned}\label{combined}
\int\limits_{0}^{T} \int\limits_{-\infty}^{\infty} [\partial_t \phi \eta(u(x,t)|\bar{u}(x+X(t),t))+\partial_x \phi q(u(x,t);\bar{u}(x+X(t),t))]\,dxdt\hspace{1in} \\
+  \int\limits_{-\infty}^{\infty}\phi(x,0)\eta(u^0(x)|\bar{u}^0(x))\,dx\hspace{2in} \\
\geq
\int\limits_{0}^{T} \int\limits_{-\infty}^{\infty}\phi\Bigg[\nabla\eta(\bar{u}(x+X(t),t))\Bigg(\partial_x\bigg|_{(x+X(t),t)}\hspace{-.45in}\big(\bar{u}(x,t)\big)\Bigg)\dot{X}(t)
\\
+\nabla\eta(\bar{u}(x+X(t),t))G(\bar{u}(\cdot,t))(x+X(t))-\nabla\eta(u(x,t))G(u(\cdot,t))(x)
\\
+\Bigg(\partial_x \bigg|_{(x+X(t),t)}\hspace{-.45in} \bar{u}^T(x,t)\Bigg)\nabla^2\eta(\bar{u}(x+X(t),t)) f(u(x,t)|\bar{u}(x+X(t),t))
\\
+\Bigg(2\partial_x\bigg|_{(x+X(t),t)}\hspace{-.45in}\bar{u}^T(x,t)\dot{X}(t)
\\
+G^T(\bar{u}(\cdot,t))(x+X(t))\Bigg)\nabla^2\eta(\bar{u}(x+X(t),t))[u(x,t)-\bar{u}(x+X(t),t)]
\\
-\nabla\eta(\bar{u}(x+X(t),t)) \Big[\Bigg(\partial_x\bigg|_{(x+X(t),t)}\hspace{-.45in}\big(\bar{u}(x,t)\big)\Bigg)\dot{X}(t)+G(\bar{u}(\cdot,t))(x+X(t))-G(u(\cdot,t))(x)\Big]
\Bigg]\,dxdt
\\
=
\int\limits_{0}^{T} \int\limits_{-\infty}^{\infty}\phi\Bigg[-\nabla\eta(u(x,t))G(u(\cdot,t))(x)
\\
+\Bigg(\partial_x \bigg|_{(x+X(t),t)}\hspace{-.45in} \bar{u}^T(x,t)\Bigg)\nabla^2\eta(\bar{u}(x+X(t),t)) f(u(x,t)|\bar{u}(x+X(t),t))
\\
+\Bigg(2\partial_x\bigg|_{(x+X(t),t)}\hspace{-.45in}\bar{u}^T(x,t)\dot{X}(t)
\\
+G^T(\bar{u}(\cdot,t))(x+X(t))\Bigg)\nabla^2\eta(\bar{u}(x+X(t),t))[u(x,t)-\bar{u}(x+X(t),t)]
\\
-\nabla\eta(\bar{u}(x+X(t),t)) \Big[-G(u(\cdot,t))(x)\Big]
\Bigg]\,dxdt.
\end{aligned}
\end{equation}

Note that we can add zero, to get
\begin{equation}
\begin{aligned}\label{note_that}
-\nabla\eta(u(x,t))G(u(\cdot,t))(x)+G^T(\bar{u}(\cdot,t))(x+X(t))\nabla^2\eta(\bar{u}(x+X(t),t))[u(x,t)-\bar{u}(x+X(t),t)]
\\
-\nabla\eta(\bar{u}(x+X(t),t)) \Big[-G(u(\cdot,t))(x)\Big]\hspace{3in}
\\
=
-G^T(u(\cdot,t))(x)\Bigg(\big(\nabla\eta(u(x,t))\big)^T-\big(\nabla\eta(\bar{u}(x+X(t),t))\big)^T
\\
-\nabla^2\eta(\bar{u}(x+X(t),t))[u(x,t)-\bar{u}(x+X(t),t)]\Bigg)
\\
+
\Bigg(G^T(\bar{u}(\cdot,t))(x+X(t))-G^T(u(\cdot,t))(x)\Bigg)\nabla^2\eta(\bar{u}(x+X(t),t))[u(x,t)-\bar{u}(x+X(t),t)]
\\
=
-G^T(u(\cdot,t))(x)(\nabla\eta(u(x,t)|\bar{u}(x+X(t),t)))^T\hspace{.84in}
\\
+
\Bigg(G^T(\bar{u}(\cdot,t))(x+X(t))-G^T(u(\cdot,t))(x)\Bigg)\nabla^2\eta(\bar{u}(x+X(t),t))[u(x,t)-\bar{u}(x+X(t),t)]
\\
=
-\nabla\eta(u(x,t)|\bar{u}(x+X(t),t))G(u(\cdot,t))(x)\hspace{1.15in}
\\
+
\Bigg(G^T(\bar{u}(\cdot,t))(x+X(t))-G^T(u(\cdot,t))(x)\Bigg)\nabla^2\eta(\bar{u}(x+X(t),t))[u(x,t)-\bar{u}(x+X(t),t)].
\end{aligned}
\end{equation}
This calculation is from \cite{VASSEUR2008323}.

Then, from \eqref{combined} and \eqref{note_that}, we get \eqref{combined1}.

\hfill\break
\uline{Step 2}
\hfill\break

Choose $0<\epsilon<\min\{T-t_0,\frac{1}{2}\delta\}$.

We apply the test function $\omega(t)\chi(x,t)$ to \eqref{combined1}, where

\begin{align}
 \omega(t)\coloneqq
  \begin{cases}
   1 & \text{if } 0\leq t< t_0\\
   \frac{1}{\epsilon}(t_0-t)+1 & \text{if } t_0\leq t < t_0+\epsilon\\
   0 & \text{if } t_0+\epsilon \leq t,
  \end{cases}
\end{align}
and
\begin{align}
 \chi(x,t)\coloneqq
  \begin{cases}
   0 & \text{if } x<h_1(t)\\
   \frac{1}{\epsilon}(x-h_1(t)) & \text{if } h_1(t)\leq x < h_1(t)+\epsilon\\
   1 & \text{if } h_1(t)+\epsilon\leq x \leq h_2(t) -\epsilon\\
   -\frac{1}{\epsilon}(x-h_2(t)) & \text{if } h_2(t)-\epsilon<x\leq h_2(t)\\
   0 & \text{if } h_2(t)<x.
  \end{cases}
\end{align}

The function $\omega$ is modeled from \cite[p.~124]{dafermos_big_book}. The function $\chi$ is from \cite[p.~765]{Leger2011_original}. We get,

\begin{equation}
\begin{aligned}\label{local_plugged_test}
&\int\limits_{0}^{t_0} \Bigg[-\int\limits_{h_1(t)}^{h_1(t)+\epsilon}\frac{1}{\epsilon}\dot{h}_1(t)\eta(u(x,t)|\bar{u}(x+X(t),t))\,dx
+\int\limits_{h_1(t)}^{h_1(t)+\epsilon}\frac{1}{\epsilon}q(u(x,t);\bar{u}(x+X(t),t))\,dx
\\
&+\int\limits_{h_2(t)-\epsilon}^{h_2(t)}\frac{1}{\epsilon}\dot{h}_2(t)\eta(u(x,t)|\bar{u}(x+X(t),t))\,dx-\int\limits_{h_2(t)-\epsilon}^{h_2(t)}\frac{1}{\epsilon}q(u(x,t);\bar{u}(x+X(t),t))\,dx\Bigg]\,dt
\\
&+
\int\limits_{h_1(0)}^{h_2(0)}\eta(u^0(x)|\bar{u}^0(x))\,dx
-
\int\limits_{t_0}^{t_0+\epsilon}\frac{1}{\epsilon}\int\limits_{h_1(t)}^{h_2(t)}\eta(u(x,t)|\bar{u}(x+X(t),t))\,dxdt
+\mathcal{O}(\epsilon)
\\
&&\hspace{-2in}\geq
\int\limits_{0}^{t_0}\int\limits_{h_1(t)}^{h_2(t)}\mbox{RHS}\,dxdt,
\end{aligned}
\end{equation}
where RHS represents everything being multiplied by $\phi$ in the integral on the right hand side of \eqref{combined1}.

We let $\epsilon\to0$ in \eqref{local_plugged_test}. We use dominated convergence, the Lebegue differentiation theorem, and recall that $u$ satisfies the strong trace property (\Cref{strong_trace_definition}). This yields,

\begin{equation}
\begin{aligned}
\int\limits_{0}^{t_0} \bigg[q(u(h_1(t)+,t);\bar{u}((h_1(t)+X(t))+,t))-q(u(h_2(t)-,t);\bar{u}((h_2(t)+X(t))-,t))
\\
+\dot{h}_2(t)\eta(u(h_2(t)-,t)|\bar{u}((h_2(t)+X(t))-,t))\hspace{2in}
\\
-\dot{h}_1(t)\eta(u(h_1(t)+,t)|\bar{u}((h_1(t)+X(t))+,t))\bigg]\,dt\hspace{1.76in}
\\
\geq
\int\limits_{h_1(t_0)}^{h_2(t_0)}\eta(u(x,t_0)|\bar{u}(x+X(t_0),t_0))\,dx
\\
-\int\limits_{h_1(0)}^{h_2(0)}\eta(u^0(x)|\bar{u}^0(x))\,dx
\\
+\int\limits_{0}^{t_0}\int\limits_{h_1(t)}^{h_2(t)}\mbox{RHS}\,dxdt,
\end{aligned}
\end{equation}
where we also used the convexity of $\eta$ to take the limit of the term
\begin{align}
\int\limits_{t_0}^{t_0+\epsilon}\frac{1}{\epsilon}\int\limits_{h_1(t)}^{h_2(t)}\eta(u(x,t)|\bar{u}(x+X(t),t))\,dxdt
\end{align}
for every $t_0$ and not just almost every $t_0$.

We receive \eqref{local_compatible_dissipation_calc}.

\end{proof}

\section{Construction of the shift}\label{construction_of_the_shift}
In this section, we prove

\begin{proposition}[Existence of the shift function]\label{systems_entropy_dissipation_room}

Fix $T>0$.   Assume $u$ is a bounded weak solution to \eqref{system}. Assume $u$ is entropic for the entropy $\eta$, and $u$ has strong traces (\Cref{strong_trace_definition}). 
 Fix $i\in\{1,n\}$. Then let $(\bar{u}_+(t),\bar{u}_-(t),\dot{s}(t))$ be an i-shock for all  $t\in[0,T)$, where $s:[0.T)\to\mathbb{R}$ is a Lipschitz continuous function. Assume also that the map $t\mapsto(\bar{u}_+(t),\bar{u}_-(t))$ is bounded.     For $i=1$, assume the hypotheses $(\mathcal{H})$ hold. Likewise, if $i=n$, assume the hypotheses $(\mathcal{H})^*$ hold.

Assume also that there exists $\rho>0$ such that for all $t\in[0,T)$ 
\begin{align}\label{gap_local_case}
r(t)>\rho,
\end{align}
where $r(t)$ satisfies $S^1_{\bar{u}_-(t)}(r(t))=\bar{u}_+(t)$.

Then, there exists a constant $a>0$ and a Lipschitz continuous map $h: [0,T)\to\mathbb{R}$ with $h(0)=s(0)$ and such that for almost every $t$,
\begin{equation}
\begin{aligned}\label{dissipation_negative_claim}
a\big(q(u_+;\bar{u}_+(t))&-\dot{h}(t)\eta(u_+|\bar{u}_+(t))\big)-q(u_-;\bar{u}_-(t))+\dot{h}(t)\eta(u_-|\bar{u}_-(t)) \leq \\
& -c \abs{\dot{s}(t)-\dot{h}(t)}^2,
\end{aligned}
\end{equation}
where $u_{\pm}\coloneqq u(u(h(t)\pm,t)$. The constants $c,a>0$ depend on $\norm{u}_{L^\infty}$, $\norm{\bar{u}_+(\cdot)}_{L^\infty([0,T))}$, $\norm{\bar{u}_-(\cdot)}_{L^\infty([0,T))}$, and $\rho$. 
\end{proposition}


%
%
%

The proof of \Cref{systems_entropy_dissipation_room} uses
\begin{proposition}
\label{dissipation_negative_theorem}
Assume the hypotheses $(\mathcal{H})$ hold. 

Let $B,\rho>0$. Then there exists a constant $a_*\in(0,1)$ depending on $B$ and $\rho$ such that the following is true:

For any $a\in(0,a_*)$, there exists a constant $c_1$ depending on $B$, $\rho$, and $a$ such that
\begin{equation}
\begin{aligned}\label{dissipation_negative}
a\big(q(S^1_u(s);S^1_{u_L}(s_R))&-\sigma^1_u(s)\eta(S^1_u(s)|S^1_{u_L}(s_R))\big)-q(u;u_L)+\sigma^1_u(s)\eta(u|u_L) \leq \\
& -c_1\abs{\sigma^1_{u_L}(s_R)-\sigma^1_u(s)}^2,
\end{aligned}
\end{equation}
for all  $u_L\in\mathcal{V}$ with $\abs{u_L}\leq B$, all $u\in\{u | \eta(u|u_L)\leq a\eta(u|S^1_{u_L}(s_R))\}$, any $s\in[0,B]$, and any $s_R\in[\rho,B]$. 

Moreover,
\begin{equation}\label{dissipation_negative_boundary_of_convex_set}
a\big(q(u;S^1_{u_L}(s_R))-\lambda_1(u)\eta(u|S^1_{u_L}(s_R))\big)-q(u;u_L)+\lambda_1(u)\eta(u|u_L)\leq -c_1,
\end{equation}
for all $u\in\{u | \eta(u|u_L)\leq a\eta(u|S^1_{u_L}(s_R))\}$ and for the same constant $c_1$.

\end{proposition}
\begin{remark}
The proof of \Cref{dissipation_negative_theorem} holds when we only have $\eta\in C^2$.
\end{remark}

 \Cref{dissipation_negative_theorem} uses ideas from the proof of Lemma 4.3 in \cite{MR3519973}, but to prove \Cref{dissipation_negative_theorem} we keep careful track of the dependencies on the constants and make sure in our calculations  to leave some extra negativity in the entropy dissipation lost at the shock $(u_L,u_R,\sigma_{L,R})$ (thus we have a negative right hand side in our \eqref{dissipation_negative} and \eqref{dissipation_negative_boundary_of_convex_set}). The idea of the extra negativity in the entropy dissipation is similar to the work \cite{2017arXiv171207348K,scalar_move_entire_solution}.

To prove \Cref{dissipation_negative_theorem}, we will need 
\begin{corollary}\label{entropy_lost_right_side_1_shock}
Assume the system \eqref{system} satisfies the hypothesis $(\mathcal{H}1)$. Fix $B,\rho>0$. Then there exists $k,\delta_0>0$ depending on $B$ and $\rho$ such that for any $\delta\in(0,\delta_0]$, $u\in\mathcal{V}\cap B_{r_0}({I_{-}})$ with $\abs{u}\leq B$ and for any $s_0\in(\rho,B)$ and $s\geq0$,
\begin{equation}
\begin{aligned}\label{entropy_lost_right_side_1_shock_inequalities}
&q(S_u^1(s);S_u^1(s_0))-\sigma_u^1(s)\eta(S_u^1(s)|S_u^1(s_0))\leq -k\abs{\sigma_u^1(s)-\sigma_u^1(s_0)}^2,\hspace{.2in}\mbox{for } \abs{s-s_0}<\delta,\\
&q(S_u^1(s);S_u^1(s_0))-\sigma_u^1(s)\eta(S_u^1(s)|S_u^1(s_0))\leq -k\delta\abs{\sigma_u^1(s)-\sigma_u^1(s_0)},\hspace{.2in}\mbox{for } \abs{s-s_0}\geq\delta.
\end{aligned}
\end{equation}
\end{corollary}
The formulas \eqref{dissipation_formula_2} and \eqref{entropy_lost_right_side_1_shock_inequalities} are modifications on a key lemma due to DiPerna \cite{MR523630}. Our proof of \Cref{entropy_lost_right_side_1_shock} is based on the proof of a very similar result in \cite[p.~387-9]{MR3519973}. We modify the proof in \cite[p.~387-9]{MR3519973} -- being careful to keep the constants $k$ and $\delta_0$ uniform in $s_0$ and $u$.

The proof of \Cref{dissipation_negative_theorem} is based on the formulas\eqref{entropy_lost_right_side_1_shock_inequalities}, and this is where the negative right hand sides in \eqref{dissipation_negative} and \eqref{dissipation_negative_boundary_of_convex_set} come from.

\Cref{entropy_lost_right_side_1_shock} itself follows from \Cref{entropy_lost_i_shock_lemma} giving us an explicit formula for the entropy lost at an entropic i-shock $(u,S^i_u(s))$, for any i-family:

\begin{lemma}\label{entropy_lost_i_shock_lemma}
For any i-shock ($i\in\{1,\ldots,n\}$) $(u,S^i_u(s),\sigma^i_u(s))$ and any $v\in\mathbb{R}^n$,
\begin{align}\label{dissipation_formula_1}
q(S^i_u(s);v)-\sigma^i_u(s)\eta(S^i_u(s)|v)=q(u;v)-\sigma^i_u(s)\eta(u|v)+\int\limits_0^s\frac{\mbox{d}}{\mbox{dt}}\sigma^i_u(t)\eta(u|S^i_u(t))\,\mbox{d}t.
\end{align}

Therefore, for any $s\geq0, s_0>0$, 
\begin{align}\label{dissipation_formula_2}
q(S^i_u(s);S^i_u(s_0))-\sigma^i_u(s)\eta(S^i_u(s)|S^i_u(s_0))=\int\limits_{s_0}^s\frac{\mbox{d}}{\mbox{dt}}\sigma^i_u(t)\Big(\eta(u|S^i_u(t))-\eta(u|S^i_u(s_0))\Big)\,\mbox{d}t.
\end{align}
\end{lemma}

See Lax \cite{MR0093653} for the formula \eqref{dissipation_formula_1}. For a proof of \eqref{dissipation_formula_1}, see \cite{MR3537479}. Note that \eqref{dissipation_formula_1} and \eqref{dissipation_formula_2} hold for a shock $(u,S^i_u(s),\sigma^i_u(s))$ from any $i$-family, $i=1,2,\ldots,n$, and not just extremal families (1-family or n-family) -- the relation \eqref{dissipation_formula_1} is a direct consequence of the Rankine-Hugoniot condition. Further, \eqref{dissipation_formula_2} comes from applying \eqref{dissipation_formula_1} twice.

\subsection{Proof of \Cref{entropy_lost_right_side_1_shock}} This is based on the proof of a similar result in \cite[p.~387-9]{MR3519973}.
\hfill \break
Define
\begin{align}
&M\coloneqq \sup_{s\in(0,B),\hspace{.05in}\abs{u}\leq B} \frac{\mbox{d}}{\mbox{d}s}\sigma^1_u(s),\\
&P\coloneqq \inf_{s\in(\rho,B),\hspace{.05in}\abs{u}\leq B} \frac{\mbox{d}}{\mbox{d}s}\eta(u|S^1_u(s)).
\end{align}

Note that by Property (a) of $(\mathcal{H}1)$ $M<0$ and by Property (b) of $(\mathcal{H}1)$ $P>0$. Furthermore, note that $M$ and $P$ depend only on the system \eqref{system}, \eqref{entropy_condition_distributional_system_chitchat}, $B$ and $\rho$.

Then by uniform continuity on the compact set $\{(s,u)|s\in[0,B] \mbox{ and } \abs{u}\leq B\}$, there exists $\delta_0>0$ such that for all $s_0\in(\rho,B)$ and for all $s\geq0$ with $\abs{s_0-s}\leq \delta_0$,
\begin{equation}
\begin{aligned}
&\abs{\frac{\mbox{d}}{\mbox{d}s}\sigma^1_u(s)-\frac{\mbox{d}}{\mbox{d}s}\sigma^1_u(s_0)}\leq \frac{1}{2}\abs{M},\\
&\abs{\frac{\mbox{d}}{\mbox{d}s}\eta(u|S^1_u(s))-\frac{\mbox{d}}{\mbox{d}s}\eta(u|S^1_u(s_0))}\leq \frac{1}{2}P,\\
\end{aligned}
\end{equation}
Note that $\delta_0$ only depends on the system \eqref{system}, \eqref{entropy_condition_distributional_system_chitchat},  $B$ and $\rho$.

In particular, 
\begin{equation}
\begin{aligned}\label{RHS_entropy_dissipation_step1}
&\abs{\frac{\mbox{d}}{\mbox{d}s}\sigma^1_u(s)-\frac{\mbox{d}}{\mbox{d}s}\sigma^1_u(s_0)}\leq \frac{1}{2}\abs{M}\leq\frac{1}{2}\abs{\frac{\mbox{d}}{\mbox{d}s}\sigma^1_u(s_0)},\\
&\abs{\frac{\mbox{d}}{\mbox{d}s}\eta(u|S^1_u(s))-\frac{\mbox{d}}{\mbox{d}s}\eta(u|S^1_u(s_0))}\leq \frac{1}{2}P\leq\frac{1}{2}\abs{\frac{\mbox{d}}{\mbox{d}s}\eta(u|S^1_u(s_0))}.
\end{aligned}
\end{equation}

From \eqref{RHS_entropy_dissipation_step1}, we get the estimates
\begin{equation}
\begin{aligned}\label{RHS_entropy_dissipation_step2}
&\frac{\mbox{d}}{\mbox{d}s}\sigma^1_u(s)=-\abs{\frac{\mbox{d}}{\mbox{d}s}\sigma^1_u(s)}\leq-\frac{1}{2}\abs{\frac{\mbox{d}}{\mbox{d}s}\sigma^1_u(s_0)},\\
&\frac{\mbox{d}}{\mbox{d}s}\eta(u|S^1_u(s))=\abs{\frac{\mbox{d}}{\mbox{d}s}\eta(u|S^1_u(s))}\geq\frac{1}{2}\abs{\frac{\mbox{d}}{\mbox{d}s}\eta(u|S^1_u(s_0))}.
\end{aligned}
\end{equation}

We use \eqref{dissipation_formula_2} and \eqref{RHS_entropy_dissipation_step2} to get for all $s$ with $\abs{s-s_0}<\delta_0$,
\begin{align}
q(S_u^1(s);S_u^1(s_0))-\sigma^1_u(s)\eta(S_u^1(s)|S_u^1(s_0))
&=
\int\limits_{s_0}^s \frac{\mbox{d}}{\mbox{d}t}\sigma_u^1(t)\Big(\eta(u|S_u^1(t))-\eta(u|S^1_u(s_0))\Big)\,dt\\
&\leq -\frac{1}{4}\abs{\frac{\mbox{d}}{\mbox{d}t}\sigma_u^1(s_0)}\frac{\mbox{d}}{\mbox{d}t}\eta(u|S_u^1(s_0))\int\limits_{s_0}^s (t-s_0)\,dt\\
&=-\frac{1}{8}\abs{\frac{\mbox{d}}{\mbox{d}t}\sigma_u^1(s_0)}\frac{\mbox{d}}{\mbox{d}t}\eta(u|S_u^1(s_0))\abs{s-s_0}^2.\end{align}

Note that due to \eqref{RHS_entropy_dissipation_step1},
\begin{align}
\abs{\frac{\mbox{d}}{\mbox{d}s}\sigma^1_u(s)}\leq\frac{3}{2}\abs{\frac{\mbox{d}}{\mbox{d}s}\sigma^1_u(s_0)}.
\end{align}
Thus,
\begin{align}
\abs{\sigma_u^1(s)-\sigma_u^1(s_0)}\leq\frac{3}{2}\abs{\frac{\mbox{d}}{\mbox{d}s}\sigma^1_u(s_0)}\abs{s-s_0},
\end{align}
which gives us that for all $s$ verifying $\abs{s-s_0}<\delta_0$,
\begin{align}
q(S_u^1(s);S_u^1(s_0))-\sigma_u^1(s)\eta(S_u^1(s)|S_u^1(s_0))\leq - k_1\abs{\sigma_u^1(s)-\sigma_u^1(s_0)}^2,
\end{align}
where we define
\begin{align}
k_1\coloneqq \frac{1}{18}P \inf_{s\in(0,B),\hspace{.05in}\abs{u}\leq B} \abs{\frac{\mbox{d}}{\mbox{d}s}\sigma^1_u(s)}^{-1}.
\end{align}
Note that $k_1$ only depends on $B$ and $\rho$.

On the other hand, we now show \eqref{entropy_lost_right_side_1_shock_inequalities} for $\abs{s-s_0}\geq\delta_0$. For all $s$ verifying $s\leq s_0-\delta_0$, we get from \eqref{dissipation_formula_2}
\begin{equation}
\begin{aligned}\label{k_2_step_1}
q(S^1_u(s);S^1_u(s_0))-\sigma^1_u(s)\eta(S^1_u(s)|S^1_u(s_0))&=\int\limits_{s}^{s_0-\delta_0}\frac{\mbox{d}}{\mbox{dt}}\sigma^1_u(t)\Big(\eta(u|S^1_u(s_0))-\eta(u|S^1_u(t))\Big)\,\mbox{d}t\\
&\hspace{.3in}+
\int\limits_{s_0-\delta_0}^{s_0}\frac{\mbox{d}}{\mbox{dt}}\sigma^1_u(t)\Big(\eta(u|S^1_u(s_0))-\eta(u|S^1_u(t))\Big)\,\mbox{d}t\\
&\coloneqq I_1+I_2.
\end{aligned}
\end{equation}

Note that for a positive constant $c_1$ satisfying
\begin{align}\label{c_1_system}
c_1\leq \inf_{s_0\in[\delta_0,B]\text{ and } \abs{u}\leq B} \Big(\eta(u|S^1_u(s_0))-\eta(u|S^1_u(s_0-\delta_0))\Big),
\end{align}
then we have (recalling Property (a) of hypothesis $(\mathcal{H}1)$)
\begin{equation}
\begin{aligned}\label{k_2_step_2}
I_1&\leq \int\limits_{s}^{s_0-\delta_0}\frac{\mbox{d}}{\mbox{dt}}\sigma^1_u(t)\Big(\eta(u|S^1_u(s_0))-\eta(u|S^1_u(s_0-\delta_0))\Big)\,\mbox{d}t\\
&\leq -c_1\abs{\sigma_u^1(s_0-\delta_0)-\sigma_u^1(s)}\\
&\leq -c_1\abs{\sigma_u^1(s_0)-\sigma_u^1(s)}+c_1\abs{\sigma_u^1(s_0)-\sigma_u^1(s_0-\delta_0)}\\
&\leq -c_1\abs{\sigma_u^1(s_0)-\sigma_u^1(s)}+c_1\delta_0\sup_{s\in(0,B),\hspace{.05in}\abs{u}\leq B} \abs{\frac{\mbox{d}}{\mbox{d}s}\sigma^1_u(s)}.
\end{aligned}
\end{equation}

Recall that $\delta_0$ depends only on $B$ and $\rho$. Thus, we can find a $c_1$ which satisfies \eqref{c_1_system}  and depends only on $B$ and $\rho$. In particular, note that 
\begin{align}\label{note_that_k2}
\delta_0 P \leq\inf_{s_0\in[\delta_0,B]\mbox{ and } \abs{u}\leq B} \Big(\eta(u|S^1_u(s_0))-\eta(u|S^1_u(s_0-\delta_0))\Big).
\end{align}
Note that  for $t\in(s_0-\delta_0,s_0)$,
\begin{align}
\eta(u|S^1_u(s_0))-\eta(u|S^1_u(t))&=\int\limits_t^{s_0} \frac{\mbox{d}}{\mbox{d}s}\eta(u|S^1_u(s))\,ds\\
&\geq P (s_0-t).
\end{align}
Thus,
\begin{equation}
\begin{aligned}\label{k_2_step_3}
I_2&\leq PM\int\limits_{s_0-\delta_0}^{s_0}(s_0-t)\,dt\\
&=\frac{\delta_0^2 PM}{2}.
\end{aligned}
\end{equation}
Recall $M<0$.

Pick 
\begin{align}
c_1\coloneqq -\delta_0 k_2,
\end{align}
where
\begin{align}
k_2\coloneqq \min\Bigg\{\frac{PM}{2\sup_{s\in(0,B),\hspace{.05in}\abs{u}\leq B} \abs{\frac{\mbox{d}}{\mbox{d}s}\sigma^1_u(s)}},P\Bigg\}.
\end{align}
Note that $k_2$ depends only on $B$ and $\rho$.

Then from \eqref{k_2_step_1},\eqref{note_that_k2}, \eqref{k_2_step_2}, and \eqref{k_2_step_3}, we get
\begin{align}
q(S^1_u(s);S^1_u(s_0))-\sigma^1_u(s)\eta(S^1_u(s)|S^1_u(s_0))\leq -\delta_0 k_2\abs{\sigma_u^1(s_0)-\sigma_u^1(s)}.
\end{align}

The case for $s>s_0+\delta_0$ is analogous to the case for $s\leq s_0-\delta_0$: For $s>s_0+\delta_0$, consider a constant $c_2>0$ such that
\begin{align}\label{c_2_system}
c_2\leq \inf_{s_0\in[\rho,B]\text{ and } \abs{u}\leq B} \Big(\eta(u|S^1_u(s_0+\delta_0))-\eta(u|S^1_u(s_0))\Big),
\end{align}
Note that $\delta_0$ only depends on $B$ and $\rho$. Thus, we can find a constant $c_2$ verifying \eqref{c_2_system} and depending only on $B$ and $\rho$. In particular, note that
\begin{align}\label{note_that_k3}
\delta_0 P \leq\inf_{s_0\in[\rho,B]\text{ and } \abs{u}\leq B} \Big(\eta(u|S^1_u(s_0+\delta_0))-\eta(u|S^1_u(s_0))\Big).
\end{align}

Then write (recalling \eqref{dissipation_formula_2}),
\begin{equation}
\begin{aligned}\label{k_3_step_1}
q(S^1_u(s);S^1_u(s_0))-\sigma^1_u(s)\eta(S^1_u(s)|S^1_u(s_0))&=\int\limits_{s_0}^{s_0+\delta_0}\frac{\mbox{d}}{\mbox{dt}}\sigma^1_u(t)\Big(\eta(u|S^1_u(t))-\eta(u|S^1_u(s_0))\Big)\,\mbox{d}t\\
&\hspace{.3in}+
\int\limits_{s_0+\delta_0}^{s}\frac{\mbox{d}}{\mbox{dt}}\sigma^1_u(t)\Big(\eta(u|S^1_u(t))-\eta(u|S^1_u(s_0))\Big)\,\mbox{d}t\\
&\coloneqq J_1+J_2.
\end{aligned}
\end{equation}

Then,
\begin{equation}
\begin{aligned}\label{k_3_step_2}
J_2&\leq \int\limits_{s_0+\delta_0}^{s}\frac{\mbox{d}}{\mbox{dt}}\sigma^1_u(t)\Big(\eta(u|S^1_u(s_0+\delta_0))-\eta(u|S^1_u(s_0))\Big)\,\mbox{d}t\\
&\leq c_2 \int\limits_{s_0+\delta_0}^{s}\frac{\mbox{d}}{\mbox{dt}}\sigma^1_u(t)\,\mbox{d}t
\\
\mbox{Then, by Property (a) of hypothesis $(\mathcal{H}1)$,}
\\
&=-c_2\abs{\sigma_u^1(s)-\sigma_u^1(s_0+\delta_0)}\\
&\leq -c_2\abs{\sigma_u^1(s)-\sigma_u^1(s_0)}+c_2\abs{\sigma_u^1(s_0+\delta_0)-\sigma_u^1(s_0)}\\
&\leq-c_2\abs{\sigma_u^1(s)-\sigma_u^1(s_0)}+ c_2\delta_0\sup_{s\in(0,B),\hspace{.05in}\abs{u}\leq B} \abs{\frac{\mbox{d}}{\mbox{d}s}\sigma^1_u(s)}.
\end{aligned}
\end{equation}

Note that  for $t\in(s_0,s_0+\delta_0)$,
\begin{align}
\eta(u|S^1_u(t))-\eta(u|S^1_u(s_0))&=\int\limits_{s_0}^{t} \frac{\mbox{d}}{\mbox{d}s}\eta(u|S^1_u(s))\,ds\\
&\geq P (t-s_0).
\end{align}
Thus,
\begin{equation}
\begin{aligned}\label{k_3_step_3}
J_1&\leq PM\int\limits_{s_0}^{s_0+\delta_0}(t-s_0)\,dt\\
&=\frac{\delta_0^2 PM}{2}.
\end{aligned}
\end{equation}
Recall $M<0$.

Pick 
\begin{align}
c_2\coloneqq -\delta_0 k_3,
\end{align}
where
\begin{align}
k_3\coloneqq \min\Bigg\{\frac{PM}{2\sup_{s\in(0,B),\hspace{.05in}\abs{u}\leq B} \abs{\frac{\mbox{d}}{\mbox{d}s}\sigma^1_u(s)}},P\Bigg\}.
\end{align}
Note that $k_3$ depends only on $B$ and $\rho$.

Then from \eqref{k_3_step_1},\eqref{note_that_k3}, \eqref{k_3_step_2}, and \eqref{k_3_step_3}, we get
\begin{align}
q(S^1_u(s);S^1_u(s_0))-\sigma^1_u(s)\eta(S^1_u(s)|S^1_u(s_0))\leq -\delta_0 k_3\abs{\sigma_u^1(s_0)-\sigma_u^1(s)}.
\end{align}

\begin{remark}
Note that in hypothesis $(\mathcal{H}1)$, we assume the 1-shock curve $S^1_u$ is parameterized by arc length. Thus, if $s<B$ then $\abs{S^1_u(s)}<B$.
\end{remark}

\subsection{Proof of \Cref{dissipation_negative_theorem}} This proof is based on the proof of Lemma 4.3 in \cite{MR3519973}.
\hfill\break

In what follows, we use $C$ to denote a generic constant which only depends on $B$ and $\rho$. 

Also, for convenience we define
\begin{align}
&u_R\coloneqq S^1_{u_L}(s_R)\\
&R_a\coloneqq \{u | \eta(u|u_L)\leq a\eta(u|u_R)\}.
\end{align}
 
\uline{Step 1} 
\hfill\break
We first need to show that for any fixed $\sigma_0\in\mathbb{R}$ such that $\lambda_1(u_L)>\sigma_0$, there exists $\beta,\epsilon_0>0$ such that
\begin{align}
\label{sigma_0_ineq}
-q(u;u_L)+\sigma_0\eta(u|u_L) \leq -\beta \eta(u|u_L),
\end{align}
for all $u\in B_{\epsilon_0}(u_L)$.

The difference between $\lambda_1(u_L)$ and $\sigma_0$ will power the proof of \eqref{dissipation_negative}. We will choose a $\sigma_0$ later. 

We use Taylor expansion to prove \eqref{sigma_0_ineq}:
\begin{align}
\label{taylor_expanse}
-q(u;u_L)+\sigma_0\eta(u|u_L) =(u-u_L)^T\nabla^2\eta(u_L)(\sigma_0I-\nabla f (u_L))(u-u_L)+\mathcal{O}(\abs{u-u_L}^3)
\end{align}

Due to the strict convexity of $\eta$, $\nabla^2\eta(u_L)$ is symmetric and strictly positive definite. Also, by assumption  $\nabla^2\eta(u_L) \nabla f(u_L)$ is symmetric. Thus these two matrices are diagonalizable in the same basis. We receive,
\begin{align}
\label{matrix_ordering}
\nabla^2\eta(u_L) \nabla f(u_L) \geq \lambda_1(u_L)\nabla^2\eta(u_L).
\end{align}

Let $C_1>0$ be a constant such that the term $\mathcal{O}(\abs{u-u_L}^3)$ in \eqref{taylor_expanse} satisfies $\mathcal{O}(\abs{u-u_L}^3)\leq C_1 \abs{u-u_L}^3$ for all $\abs{u_L}\leq B$ and all $u\in B_{1}(u_L)$. Note $C_1$ depends only on $B$. Let 
\begin{align}
C_2\coloneqq \inf_{\abs{x}=1\hspace{.07cm},\hspace{.07cm}\abs{u_L}\leq B}x^T\nabla^2\eta(u_L)x.
\end{align}
Note that because $\eta$ is strictly convex, $C_2>0$. Note $C_2$ depends only on $B$.

Then, for all 
\begin{align}
\epsilon_0<\min\{\frac{C_2}{2C_1}(\lambda_1(u_L)-\sigma_0),1\}
\end{align}
and for all $u\in B_{\epsilon_0}(u_L)$, we have from \eqref{matrix_ordering} and because $\lambda_1(u_L)>\sigma_0$, 
\begin{align}
-q(u;u_L)+\sigma_0\eta(u|u_L) &\leq -(\lambda_1(u_L)-\sigma_0)(u-u_L)^T\nabla^2\eta(u_L)(u-u_L)+\mathcal{O}(\abs{u-u_L}^3)\\
&\leq -\frac{(\lambda_1(u_L)-\sigma_0)}{2}(u-u_L)^T\nabla^2\eta(u_L)(u-u_L)\\
&\leq -C\frac{(\lambda_1(u_L)-\sigma_0)}{2}\eta(u|u_L)
\end{align}
by \Cref{entropy_relative_L2_control_system}.  This proves \eqref{sigma_0_ineq}, with 
\begin{align}
\label{beta_value}
\beta=C\frac{(\lambda_1(u_L)-\sigma_0)}{2}.
\end{align}

\uline{Step 2}
\hfill\break
We can now compute to show \eqref{dissipation_negative}.

In the context of \Cref{entropy_lost_right_side_1_shock}, we can use the same value of $B$ in \Cref{entropy_lost_right_side_1_shock} as in \Cref{dissipation_negative_theorem}. In \Cref{entropy_lost_right_side_1_shock}, we have constants $k$ and $\delta_0$. Note that these constants depend on $B$ and $\rho$. In the context of \Cref{entropy_lost_right_side_1_shock}, we are allowed to choose $\delta$ as long as it is sufficiently small. Choose 
\begin{align}\label{choose_delta}
\delta\coloneqq \min\{\delta_0,\frac{s_R}{2}\}
\end{align}
for the $\delta$ in \Cref{entropy_lost_right_side_1_shock}. Note that $\delta$ depends on $B$ and $\rho$.
Then, define 
\begin{align}
k^*\coloneqq \min\{\delta k,k\}.
\end{align}

Note that $k^*$ depends on $B$ and $\rho$.

Define the following quantities,
\begin{align}
M&\coloneqq \sup_{0\leq s \leq B\hspace{.07cm},\hspace{.1cm}\abs{u}\leq B+1} \frac{d}{ds}\sigma_u^1(s)\label{M_def},
\shortintertext{where the constant $M$ exists and satisfies $M<0$ because by the hypotheses $(\mathcal{H}1)$,  $(s,u)\mapsto \sigma_u^1(s)$ is $C^1$ and $\frac{d}{ds}\sigma_u^1(s)<0$.  We further define,}
L&\coloneqq \sup_{\abs{u}\leq B+1}\norm{\nabla \lambda_1},\\
\sigma_0&\coloneqq \lambda_1(u_L)+\frac{k^{*}M}{16C_3}\frac{s_R}{2}\label{sigma_0_def},
\end{align}
where $C_3$ will appear later, in \eqref{c_2_eq} -- and $C_3$ will depend on $B$.  The constant $L$ exists because by assumption the flux $f\in\ C^2(\mathcal{V})$ (see the remarks after the hypotheses $(\mathcal{H})$ and $(\mathcal{H})^*$). Note $M$ and $L$ depend only on $B$. 

We choose $\epsilon_0$ such that 
\begin{align}\label{condition_1_epsilon0}
\epsilon_0<\min\Bigg\{-\frac{k^{*}M}{16C_3}\frac{s_R}{2}\frac{1}{L},-\frac{C_2}{C_1}\frac{k^{*}M}{16C_3}\frac{s_R}{2}\frac{1}{L},-\frac{C_2}{C_1}\frac{k^{*}M}{16C_3}\frac{s_R}{2},1\Bigg\}.
\end{align}
Note the right hand side of \eqref{condition_1_epsilon0} depends on $B$ and $\rho$. We also need to make sure that $a_*$ is small enough such that $R_a\subset B_{\epsilon_0}(u_L)$ for all $0<a<a_*$. Recall \eqref{cond_a}.

We claim that for all $u\in B_{\epsilon_0}(u_L)$,
\begin{align}
\sigma_u^1(s)\leq \sigma_0, \hspace{.1in} \mbox{for } s\geq \frac{s_R}{2},\label{claim_1}
\end{align}
and
\begin{align}
\lambda_1(u)-\sigma_0\leq \frac{k^{*}}{8C_3}\abs{\sigma_u^1(\frac{s_R}{2})-\sigma_u^1(s_R)}.\label{claim_2}
\end{align}

We show \eqref{claim_1}: for $s\geq \frac{s_R}{2}$,
\begin{align}
\sigma_u^1(s) &\leq \sigma_u^1(0)+sM\\
&=\lambda_1(u)+sM\\
&\leq \lambda_1(u_L)-\frac{k^*M}{16C_3}\frac{s_R}{2}+sM\\
&=\lambda_1(u_L)+M(s-\frac{k^*}{16C_3}\frac{s_R}{2})\\
&\leq \lambda_1(u_L)+M(\frac{s_R}{2}-\frac{k^*}{16C_3}\frac{s_R}{2})\\
&= \lambda_1(u_L)+M\frac{s_R}{2}(1-\frac{k^*}{16C_3})\\
&< \sigma_0,
\end{align}
where to get the last inequality we can make $C_3$ larger if necessary such that $\frac{k^*}{16C_3}<\frac{1}{2}$, noting $C_3$ will then depend on $\rho$ and $B$.

We now show \eqref{claim_2}:
\begin{align}
\lambda_1(u)-\sigma_0 &\leq  \lambda_1(u_L)-\frac{k^*M}{16C_3}\frac{s_R}{2} -\sigma_0\\
&= \lambda_1(u_L)-\frac{k^*M}{16C_3}\frac{s_R}{2}-\lambda_1(u_L)-\frac{k^*M}{16C_3}\frac{s_R}{2}\\
&=-\frac{k^*M}{8C_3}\frac{s_R}{2} \leq \frac{k^*}{8C_3}\abs{\sigma_u^1(\frac{s_R}{2})-\sigma_u^1(s_R)},
\end{align}
by definition of $M$.

\vspace{.05in}
\textbf{To prove \eqref{dissipation_negative}, we consider two cases: $s\geq \frac{s_R}{2}$ and $s < \frac{s_R}{2}$.}
\vspace{.05in}

We first consider $s\geq \frac{s_R}{2}$. From \eqref{48_kang_vasseur_a_contraction}, we get

\begin{align}
q(&S^1_u(s);u_R)-\sigma^1_u(s)\eta(S^1_u(s)|u_R)
=
-\big(q(u_R;S_u^1(s_R))-\sigma^1_u(s)\eta(u_R|S^1_u(s_R))\big)\\
&+\big(q(S^1_u(s);S^1_u(s_R))-\sigma^1_u(s)\eta(S^1_u(s)|S^1_u(s_R))\big)\\
&+\big(\nabla\eta(u_R)-\nabla\eta(S^1_u(s_R))\big)\big(f(u_R)-f(S^1_u(s))-\sigma_u^1(s)(u_R-S^1_u(s))\big).
\end{align}

By using the Rankine-Hugoniot jump compatibility conditions
\begin{align}
f(u_R)-f(u_L)=\sigma_{u_L}^1(s_R)(u_R-u_L),\\
f(S^1_u(s))-f(u)=\sigma_u^1(s)(S^1_u(s)-u),
\end{align} 
we can rewrite

\begin{align}
q(&S^1_u(s);u_R)-\sigma^1_u(s)\eta(S^1_u(s)|u_R)
=
-\big(q(u_R;S_u^1(s_R))-\sigma^1_u(s)\eta(u_R|S^1_u(s_R))\big)\\
&+\big(q(S^1_u(s);S^1_u(s_R))-\sigma^1_u(s)\eta(S^1_u(s)|S^1_u(s_R))\big)\\
&+\big(\nabla\eta(u_R)-\nabla\eta(S^1_u(s_R))\big)\big(f(u_L)-f(u)-\sigma_u^1(s)(u_L-u)\\
&\hspace{.8in}+(\sigma_{u_L}^1(s_R)-\sigma_u^1(s))(u_R-u_L)\big)\\
&\coloneqq I_1+I_2+I_3.
\end{align}

To estimate $I_2$ and $I_3$, we use the following rough estimates. In these estimates, the constants are uniform in $u_L$ (with $\abs{u_L}\leq B$) and $s_R\in[\rho,B]$. The estimates hold for any $u\in B_{\epsilon_0}(u_L)$ (recall by \eqref{condition_1_epsilon0}, $\epsilon_0<1$).
Recall that by the hypothesis $(\mathcal{H}1) $, $(s,u)\mapsto S_u^1(s)$ is $C^1$. Then, 
\begin{equation}
\begin{aligned}\label{estimates}
\abs{\eta(S^1_{u_L}(s_R)|S^1_{u}(s_R))}&\leq C \abs{S^1_{u_L}(s_R)-S^1_{u}(s_R)}^2\leq C \abs{u_L-u}^2,
\\
&\hspace{-1.7in}\mbox{because $\eta\in C^2$ and by \Cref{entropy_relative_L2_control_system}, $\eta(a|b)$ is locally quadratic in $a-b$. Continuing,}
\\
\abs{(q(S^1_{u_L}(s_R);S^1_{u}(s_R))}&\leq C \abs{S^1_{u_L}(s_R)-S^1_{u}(s_R)}^2\leq C \abs{u_L-u}^2,
\\
&\hspace{-1.7in}\mbox{because $q\in C^2$ and $q(a;b)$ is locally quadratic in $a-b$. Further,}
\\
\abs{\nabla\eta(S^1_{u_L}(s_R))-\nabla\eta(S^1_{u}(s_R))}&\leq C \abs{S^1_{u_L}(s_R)-S^1_{u}(s_R)}\leq C \abs{u_L-u},
\\
&\hspace{-1.7in}\mbox{because $\eta\in C^2(\mathcal{V})$. Lastly,}
\\ 
\abs{\sigma^1_{u_L}(s_R)-\sigma_u^1(s_R)}&\leq C \abs{u_L-u},
\\
&\hspace{-1.7in}\mbox{because by the hypothesis $(\mathcal{H}1) $,  $(s,u)\mapsto \sigma_u^1(s)$ is $C^1$.}
\end{aligned}
\end{equation}

Then, from the estimates \eqref{estimates}, we get
\begin{equation}
\begin{aligned}\label{I_1_systems_estimate}
I_1&=-q(u_R;S_u^1(s_R))+\sigma^1_u(s)\eta(S^1_{u_L}(s_R)|S^1_u(s_R))\\
&=-q(u_R;S_u^1(s_R))+\sigma^1_u(s_R)\eta(S^1_{u_L}(s_R)|S^1_u(s_R))+(\sigma^1_u(s)-\sigma^1_u(s_R))\eta(S^1_{u_L}(s_R)|S^1_u(s_R))\\
&\leq C\abs{u_L-u}^2(1+\abs{\sigma^1_u(s)-\sigma^1_u(s_R)}),
\end{aligned}
\end{equation}
and
\begin{equation}
\begin{aligned}\label{I_3_systems_estimate}
I_3&=\big(\nabla\eta(u_R)-\nabla\eta(S^1_u(s_R))\big)\big(f(u_L)-f(u)-\sigma_u^1(s)(u_L-u)\\&\hspace{.8in}+(\sigma_{u_L}^1(s_R)-\sigma_u^1(s))(u_R-u_L)\big)\\
&\leq C\abs{u_L-u}(\abs{u_L-u}+\abs{\sigma^1_u(s)-\sigma^1_u(s_R)}\abs{u_L-u}+\abs{\sigma^1_u(s)-\sigma^1_u(s_R)}).
\end{aligned}
\end{equation}

To control $I_2$, we use \Cref{entropy_lost_right_side_1_shock}. Note first that

\begin{equation}
\begin{aligned}\label{estimates1}
\abs{\sigma_u^1(s)-\sigma_{u_L}^1(s_R)}^2&\leq \Big(\abs{\sigma_u^1(s)-\sigma_{u}^1(s_R)}+\abs{\sigma_u^1(s_R)-\sigma_{u_L}^1(s_R)}\Big)^2\\
&\leq \Big(\abs{\sigma_u^1(s)-\sigma_{u}^1(s_R)}+C\abs{u-u_L}\Big)^2\\
&=\abs{\sigma_u^1(s)-\sigma_{u}^1(s_R)}^2+2C\abs{u-u_L}\abs{\sigma_u^1(s)-\sigma_{u}^1(s_R)}+C^2\abs{u-u_L}^2.
\end{aligned}
\end{equation}

Then, for $\abs{s-s_R}<\delta$ we use \Cref{entropy_lost_right_side_1_shock} and \eqref{estimates1} above:
\begin{align}
I_2&=q(S^1_u(s);S^1_u(s_R))-\sigma^1_u(s)\eta(S^1_u(s)|S^1_u(s_R))\\
&\leq -k^* \abs{\sigma_u^1(s)-\sigma_u^1(s_R)}^2\\
&=-\frac{k^*}{2} \abs{\sigma_u^1(s)-\sigma_u^1(s_R)}^2-\frac{k^*}{2} \abs{\sigma_u^1(s)-\sigma_u^1(s_R)}^2\\
&\leq -\frac{k^*}{2} \abs{\sigma_u^1(s)-\sigma_u^1(s_R)}^2 -\frac{k^*}{2}\abs{\sigma_u^1(s)-\sigma_{u_L}^1(s_R)}^2 \\
&\hspace{.5in}+Ck^*\abs{u-u_L}\abs{\sigma_u^1(s)-\sigma_u^1(s_R)}+\frac{k^*}{2}C^2\abs{u-u_L}^2\\
&= -\frac{k^*}{2} \abs{\sigma_u^1(s)-\sigma_u^1(s_R)}^2 -\frac{k^*}{2}\abs{\sigma_u^1(s)-\sigma_{u_L}^1(s_R)}^2 \\
&\hspace{.5in}+C\abs{u-u_L}\abs{\sigma_u^1(s)-\sigma_u^1(s_R)}+C\abs{u-u_L}^2,
\end{align}
where in the last equality we just absorb some constants into the $C$.
 
Then, if $\abs{s-s_R}<\delta$, we use our estimates on $I_1, I_2$, and $I_3$ to get
\begin{equation}
\begin{aligned}\label{estimates2}
q(S^1_u(s);u_R)-&\sigma^1_u(s)\eta(S^1_u(s)|u_R)
\leq 
 -\frac{k^*}{2} \abs{\sigma_u^1(s)-\sigma_u^1(s_R)}^2 -\frac{k^*}{2}\abs{\sigma_u^1(s)-\sigma_{u_L}^1(s_R)}^2 \\&+C\abs{u-u_L}\abs{\sigma_u^1(s)-\sigma_u^1(s_R)}+C\abs{u-u_L}^2\abs{\sigma_u^1(s)-\sigma_u^1(s_R)}+C\abs{u-u_L}^2\\
 &\leq -\frac{k^*}{2}\abs{\sigma_u^1(s)-\sigma_{u_L}^1(s_R)}^2 +C(\abs{u-u_L}^2+\abs{u-u_L}^4),
\\
&\hspace{-1.04in}\mbox{where we have used the version of Young's inequality with $\epsilon$. Continuing,}
\\
 &\leq  -\frac{k^*}{2}\abs{\sigma_u^1(s)-\sigma_{u_L}^1(s_R)}^2 +C\abs{u-u_L}^2,
 \end{aligned}
\end{equation}
because $u\in B_{\epsilon_0}(u_L)$ and by \eqref{condition_1_epsilon0}, $\epsilon_0<1$.

Thus, putting everything together, we have for $s\geq\frac{s_R}{2}$ and $\abs{s-s_R}<\delta$,
\begin{align}
a&\big(q(S^1_u(s);u_R)-\sigma^1_u(s)\eta(S^1_u(s)|u_R)\big)-q(u;u_L)+\sigma^1_u(s)\eta(u|u_L)\\
&\leq  aC\abs{u-u_L}^2 -\frac{ak^*}{2}\abs{\sigma_u^1(s)-\sigma_{u_L}^1(s_R)}^2 -q(u;u_L)+\sigma_0\eta(u|u_L),
\shortintertext{by \eqref{estimates2} and \eqref{claim_1}. Continuing,}
&\leq aC\abs{u-u_L}^2 -\frac{ak^*}{2}\abs{\sigma_u^1(s)-\sigma_{u_L}^1(s_R)}^2 -\beta\eta(u|u_L),
\end{align}
by \eqref{sigma_0_ineq}. We recall \Cref{entropy_relative_L2_control_system}, and choose $a_*$ small enough such that $aC\abs{u-u_L}^2- \beta\eta(u|u_L)\leq 0$ for all $u$. As always, we also require that $a_*$ is small enough such that $R_a\subset B_{\epsilon_0}(u_L)$ for all $0<a<a_*$ (recall the condition \eqref{cond_a}).
This proves \eqref{dissipation_negative}.

When $s\geq \frac{s_R}{2}$ and $\abs{s-s_R}>\delta$, using \Cref{entropy_lost_right_side_1_shock} and our estimates on $I_1$ and $I_3$ \eqref{I_1_systems_estimate} and \eqref{I_3_systems_estimate},

\begin{equation}
\begin{aligned}\label{estimates1_redo_control}
q(S^1_u(s);u_R)-&\sigma^1_u(s)\eta(S^1_u(s)|u_R)
\leq 
-k^* \abs{\sigma_u^1(s)-\sigma_u^1(s_R)} +C\abs{u-u_L}\abs{\sigma_u^1(s)-\sigma_u^1(s_R)}\\&+C\abs{u-u_L}^2\abs{\sigma_u^1(s)-\sigma_u^1(s_R)}+C\abs{u-u_L}^2\\
&= -\frac{k^*}{2} \abs{\sigma_u^1(s)-\sigma_u^1(s_R)} -\frac{k^*}{2}\abs{\sigma_u^1(s)-\sigma_{u}^1(s_R)} \\&+C\abs{u-u_L}\abs{\sigma_u^1(s)-\sigma_u^1(s_R)}+C\abs{u-u_L}^2\abs{\sigma_u^1(s)-\sigma_u^1(s_R)}+C\abs{u-u_L}^2\\
 &\leq-\frac{k^*}{2} \abs{\sigma_u^1(s)-\sigma_u^1(s_R)} +C\abs{u-u_L}^2,
\end{aligned}
\end{equation}
because $u\in B_{\epsilon_0}(u_L)$ and we pick $\epsilon_0$ even smaller such that $\epsilon_0<\min\{\frac{k^*}{4C},1\}$. Recall we require that $a_*$ is small enough such that $R_a\subset B_{\epsilon_0}(u_L)$ for all $0<a<a_*$ (see \eqref{cond_a}).

Putting everything together, for $s \geq \frac{s_R}{2}$ and $\abs{s-s_R}>\delta$,
\begin{align}
a&\big(q(S^1_u(s);u_R)-\sigma^1_u(s)\eta(S^1_u(s)|u_R)\big)-q(u;u_L)+\sigma^1_u(s)\eta(u|u_L)\\
&\leq  aC\abs{u-u_L}^2 -\frac{ak^*}{2}\abs{\sigma_u^1(s)-\sigma_{u}^1(s_R)} -q(u;u_L)+\sigma_0\eta(u|u_L)
\shortintertext{by \eqref{estimates1_redo_control} and \eqref{claim_1}. Continuing,}
&\leq aC\abs{u-u_L}^2 -\frac{ak^*}{2}\abs{\sigma_u^1(s)-\sigma_{u}^1(s_R)} -\beta\eta(u|u_L)
\end{align}
by \eqref{sigma_0_ineq}. We again recall \Cref{entropy_relative_L2_control_system}, and choose $a_*$ small enough such that $aC\abs{u-u_L}^2- \beta\eta(u|u_L)\leq 0$ for all $u$. Recall, we always require that $a_*$ is small enough such that $R_a\subset B_{\epsilon_0}(u_L)$ for all $0<a<a_*$ (use condition \eqref{cond_a}).
Again note that with $\sigma_0$ defined in \eqref{sigma_0_def} and $\beta$ defined in \eqref{beta_value}, $\beta=Cs_R$. Finally, we get the right hand side of \eqref{dissipation_negative} by noting that $\abs{\sigma_u^1(s)-\sigma_{u}^1(s_R)}$ will be uniformly bounded from below for all $\abs{s-s_R}>\delta$ (with $s\in[0,B]$ and $s_R\in[\rho,B]$), because by Property (a) of $(\mathcal{H}1)$, $\frac{d}{ds}\sigma^1_u(s)<0$. Furthermore, the term $\abs{\sigma^1_{u_L}(s_R)-\sigma^1_u(s)}^2$ on the right hand side of \eqref{dissipation_negative} is bounded (with the bound depending on $B$). Thus, by making $c_1$ sufficiently small, this proves \eqref{dissipation_negative}. Recall also that $\delta$ depends on $B$ and $\rho$. Thus, $c_1$ depends on $B$ and $\rho$.

On the other hand, we now consider $s< \frac{s_R}{2}$. From \eqref{choose_delta}, we have $\delta<\frac{s_R}{2}$. Thus when $s< \frac{s_R}{2}$, $\abs{s-s_R}>\delta$.

The computations in \eqref{estimates1_redo_control} apply exactly. We get again,

\begin{align}
q(S^1_u(s);u_R)-\sigma^1_u(s)\eta(S^1_u(s)|u_R)
\leq -\frac{k^*}{2} \abs{\sigma_u^1(s)-\sigma_u^1(s_R)} +C\abs{u-u_L}^2,
\end{align}
again because $u\in B_{\epsilon_0}(u_L)$ and $\epsilon_0$ verifies $\epsilon_0<\frac{k^*}{4C}$. 

Then, because by the assumptions $(\mathcal{H})$ $\frac{d}{ds}\sigma^1_u(s)<0$, we have for all $s<\frac{s_R}{2}$,
\begin{align}
q(S^1_u(s);u_R)-\sigma^1_u(s)\eta(S^1_u(s)|u_R)
&\leq -\frac{k^*}{2} \abs{\sigma_u^1(\frac{s_R}{2})-\sigma_u^1(s_R)} +C\abs{u-u_L}^2
\shortintertext{Then, for $\epsilon_0$ small enough such that $C\epsilon_0^2\leq\frac{k^*Ms_R}{8}$ (where $M$ is from \eqref{M_def}),}
&\leq -\frac{k^*}{4} \abs{\sigma_u^1(\frac{s_R}{2})-\sigma_u^1(s_R)}. 
\end{align}
Recall we also need $a_*$ sufficiently small so that $R_a\subset B_{\epsilon_0}(u_L)$ for all $0<a<a_*$.  See \eqref{cond_a}. 

To control the left hand side of the entropy dissipation in \eqref{dissipation_negative}, we estimate
\begin{equation}
\begin{aligned}\label{c_2_eq}
-q(u;u_L)+\sigma^1_u(s)\eta(u|u_L)&\leq -q(u;u_L)+\lambda_1(u)\eta(u|u_L),
\\
&\hspace{-1.55in}\mbox{because by the assumptions $(\mathcal{H})$ $\frac{d}{ds}\sigma^1_u(s)<0$ and $\sigma^1_u(0)=\lambda_1(u)$. Continuing, }
\\
&= -q(u;u_L)+\sigma_0\eta(u|u_L)+(\lambda_1(u)-\sigma_0)\eta(u|u_L)\\
&\leq (\lambda_1(u)-\sigma_0)\eta(u|u_L),
\\
&\hspace{-1.55in}\mbox{by \eqref{sigma_0_ineq}. Continuing, }
\\
&\leq a(\lambda_1(u)-\sigma_0)\eta(u|u_R),
\\
&\hspace{-1.55in}\mbox{because $u\in R_a\subset B_{\epsilon_0}(u_L)$. Continuing, recall $\epsilon_0<1$ by \eqref{condition_1_epsilon0}. Furthermore,}\\ 
&\hspace{-1.55in}\mbox{recall \Cref{entropy_relative_L2_control_system}, $\abs{u_L}\leq B$, $s_R\leq B$, and $S^1_{u_L}$ is parameterized by arc length.}\\ 
&\hspace{-1.55in}\mbox{Then, we get\hspace{5in} }
\\
&\leq aC_3(\lambda_1(u)-\sigma_0).
\end{aligned}
\end{equation}
Note $C_3$ is a constant which depends on $B$.

Putting everything together, for all $s<\frac{s_R}{2}$, 
\begin{equation}
\begin{aligned}\label{ending_computation}
a&\big(q(S^1_u(s);u_R)-\sigma^1_u(s)\eta(S^1_u(s)|u_R)\big)-q(u;u_L)+\sigma^1_u(s)\eta(u|u_L)\\
&\leq -a\Big(\frac{k^*}{4} \abs{\sigma_u^1(\frac{s_R}{2})-\sigma_u^1(s_R)}-C_3(\lambda_1(u)-\sigma_0)\Big)\\
&\leq-a\Big(\frac{k^*}{8} \abs{\sigma_u^1(\frac{s_R}{2})-\sigma_u^1(s_R)}\Big),
\\
\mbox{by \eqref{claim_2}. Continuing,}
\\
&\leq \frac{aMk^* s_R}{16},
\end{aligned}
\end{equation}
where $M$ is from \eqref{M_def}. Recall $M<0$.

Note that the term $\abs{\sigma^1_{u_L}(s_R)-\sigma^1_u(s)}^2$ on the right hand side of \eqref{dissipation_negative} is bounded (with the bound depending on $B$), so we get the right hand side of \eqref{dissipation_negative} by making $c_1$ smaller if necessary. Note that in making this adjustment to $c_1$, $c_1$ will depend on $B$ and $\rho$. This proves \eqref{dissipation_negative}.

Lastly, we get \eqref{dissipation_negative_boundary_of_convex_set} by the same computation \eqref{ending_computation} and taking $s=0$. Recall that by the hypothesis $(\mathcal{H}1)$, $\sigma^1_u(0)=\lambda_1(u)$.

\subsection{Proof of \Cref{systems_entropy_dissipation_room}}
By the remark about taking the negative of the flux ($-f$) if necessary, we can assume that $(\bar{u}_+(t),\bar{u}_-(t),\dot{s}(t))$ is a 1-shock. 

We will use \Cref{dissipation_negative_theorem}. The 1-shock $(\bar{u}_+(t),\bar{u}_-(t),\dot{s}(t))$ in \Cref{systems_entropy_dissipation_room} will play the role of $(u_L,S^1_{u_L}(s_R))$ in \Cref{dissipation_negative_theorem}. Take $R\coloneqq \max\{\norm{u}_{L^\infty},\norm{\bar{u}_-(\cdot)}_{L^\infty([0,T))}\}$ and then take the $\tilde{S}$ corresponding to this $R$ as in Property (c) of $(\mathcal{H}1)$. Define the $B$ in \Cref{dissipation_negative_theorem} to be $B\coloneqq \max\{R,\tilde{S},\norm{\bar{u}_+(\cdot)}_{L^\infty([0,T))}\}$. Then, we have that for all $(u_-,u_+,\sigma)$ 1-shock with $u_+,u_- < R$, there exists $s\in(0,B)$ such that $u_+=S^1_{u_-}(s)$. Further, note that $B$ depends on $\norm{u}_{L^\infty}$ and $\norm{\bar{u}_-(\cdot)}_{L^\infty([0,T))}$.

Then,  pick $0<a<1$ as in \Cref{dissipation_negative_theorem}. Here, $a$ is playing the same role as the $a$ in \Cref{dissipation_negative_theorem}. 

Throughout this proof, $c$ denotes a generic constant that depends on $\norm{u}_{L^\infty}$, $\rho$, $\norm{\bar{u}_+(\cdot)}_{L^\infty([0,T))}$, $\norm{\bar{u}_-(\cdot)}_{L^\infty([0,T))}$, and $a$.

Note by \Cref{dissipation_negative_theorem}, the constant $a$ depends on $\norm{u}_{L^\infty}$, $\norm{\bar{u}_-(\cdot)}_{L^\infty([0,T))}$,$\norm{\bar{u}_+(\cdot)}_{L^\infty([0,T))}$, and $\rho$.

\uline{Step 1}

We now show that for any $\gamma_0>0$,

\begin{align}\label{bound_on_inf}
\inf \eta(u|u_L)- a\eta(u|u_R) \geq c_4\gamma_0^2
\end{align}
for a constant $c_4>0$, where the infimum runs over all $(u,u_L,u_R)$ such that $\mbox{dist}(u,\{w|\eta(w|u_L)\leq a\eta(w|u_R)\})\geq \gamma_0$ and $\abs{u_L},\abs{u_R}\leq B$. Here, $B$ is from \Cref{dissipation_negative_theorem} and the distance $\mbox{dist}(x,A)$ between a point $x$ and a set $A$ is defined in the usual way,
\begin{align}
\mbox{dist}(x,A) \coloneqq \inf_{y\in A} \abs{x-y}.
\end{align}

Consider any triple $(u,u_L,u_R)$ such that  $\mbox{dist}(u,\{w|\eta(w|u_L)\leq a\eta(w|u_R)\})\geq \gamma_0$ and $\abs{u_L},\abs{u_R}\leq B$. 

By \Cref{dissipation_negative_theorem}, the set $\{w|\eta(w|u_L)\leq a\eta(w|u_R)\}$ is compact. Thus, there exists $w_0\in\{w|\eta(w|u_L)\leq a\eta(w|u_R)\}$ such that
\begin{align}
\abs{u-w_0}= \mbox{dist}(u,\{w|\eta(w|u_L)\leq a\eta(w|u_R)\}).
\end{align}

We Taylor expand the function 
\begin{align}
\Gamma(u)\coloneqq \eta(u|u_L)- a\eta(u|u_R)
\end{align}
around the point $w_0$:

\begin{align}
\Gamma(u)=\Gamma(w_0)+\nabla\Gamma(w_0)(u-w_0)+\int\limits_0^1 (1-t)(u-w_0)^{T}\nabla^2\Gamma(w_0+t(u-w_0))(u-w_0)\,dt.
\end{align}

By definition of $w_0$, we must have $\Gamma(w_0)=0$ and $\nabla\Gamma(w_0)(u-w_0)\geq0$.

Note that $\nabla^2\Gamma=(1-a)\nabla^2\eta$. Thus, by strict convexity of $\eta$ and because $0<a<1$, we have $\nabla^2\Gamma\geq cI$ for some constant $c>0$.

We then calculate,
\begin{align}
&\int\limits_0^1 (1-t)(u-w_0)^{T}\nabla^2\Gamma(w_0+t(u-w_0))(u-w_0)\,dt\\
&\geq \int\limits_0^{.5} (1-t)(u-w_0)^{T}\nabla^2\Gamma(w_0+t(u-w_0))(u-w_0)\,dt,
\shortintertext{where we have changed the limits of integration. Continuing,}
&\geq .5 c \abs{u-w_0}^2 \geq .5 c \gamma_0^2,
\end{align}
where the last inequality comes from $\mbox{dist}(u,\{w|\eta(w|u_L)\leq a\eta(w|u_R)\})\geq \gamma_0$. This proves \eqref{bound_on_inf}. 

We choose 
\begin{align}\label{epsilon_0_def}
\gamma_0\coloneqq \frac{c_1}{2L_*},
\end{align}
where $c_1$ is from \Cref{dissipation_negative_theorem} and  $L_*$ is the Lipschitz constant of the map
\begin{align}
(u,u_L,u_R)\mapsto a\big(q(u;u_R)-\lambda_1(u)\eta(u|u_R)\big)-q(u;u_L)+\lambda_1(u)\eta(u|u_L).
\end{align}

\uline{Step 2}

Define
\begin{align}\label{V_def_systems}
V(u,t)\coloneqq \lambda_{1}(u)-C_*\mathbbm{1}_{\{u|a\eta(u|\bar{u}_+(t))<\eta(u|\bar{u}_-(t))\}}(u),
\end{align}
where $C_*>0$ is a large constant, which we can pick to be
\begin{align}\label{C_star_def}
C_*\coloneqq \frac{1}{c_4\gamma_0^2}\Bigg(\sup_{u,u_L,u_R\in B_{B}(0)}\abs{aq(u;u_R)-q(u;u_L)}+1\Bigg) + 2\sup_{u\in B_{B}(0)}\abs{\lambda_1(u)},
\end{align}
where $c_4$ is from \eqref{bound_on_inf}.

We solve the following ODE in the sense of Filippov flows,
\begin{align}
  \begin{cases}\label{ODE}
   \dot{h}(t)=V(u(h(t),t),t)\\
   h(0)=s(0),
  \end{cases}
\end{align}

The existence of such an $h$ comes from the following lemma,
\begin{lemma}[Existence of Filippov flows]\label{Filippov_existence}
Let $V(u,t):\mathbb{R}^n \times [0,\infty)\to\mathbb{R}$ be bounded on $\mathbb{R}^n \times [0,\infty)$, upper semi-continuous in $u$, and measurable in $t$. Let $u$ be a bounded, weak solution to \eqref{system}, entropic for the entropy $\eta$. Assume also that $u$ verifies the strong trace property (\Cref{strong_trace_definition}). Let $x_0\in\mathbb{R}$. Then we can solve 
\begin{align}
  \begin{cases}\label{ODE}
   \dot{h}(t)=V(u(h(t),t),t)\\
   h(0)=x_0,
  \end{cases}
\end{align}
in the Filippov sense. That is, there exists a Lipschitz function $h:[0,\infty)\to\mathbb{R}$ such that
\begin{align}
\mbox{Lip}[h]\leq \norm{V}_{L^\infty},\label{fact1}\\
h(0)=x_0,\label{fact2}
\shortintertext{and}
\dot{h}(t)\in I[V(u_+,t),V(u_-,t)],\label{fact3}
\end{align}
for almost every $t$, where $u_\pm\coloneqq u(h(t)\pm,t)$ and $I[a,b]$ denotes the closed interval with endpoints $a$ and $b$. 

Moreover, for almost every $t$,
\begin{align}
f(u_+)-f(u_-)=\dot{h}(u_+-u_-),\label{fact4}\\
q(u_+)-q(u_-)\leq\dot{h}(\eta(u_+)-\eta(u_-)),\label{fact5}
\end{align}
which means that for almost every $t$, either $(u_+,u_-,\dot{h})$ is an entropic shock (for $\eta$) or $u_+=u_-$.
\end{lemma}

The proof of \eqref{fact1}, \eqref{fact2}, and \eqref{fact3} is very similar to the proof of Proposition 1 in \cite{Leger2011}. A proof of \eqref{fact1}, \eqref{fact2}, and \eqref{fact3} is included in \Cref{Filippov_existence_section} for the reader's convenience.

It is well known that \eqref{fact4} and \eqref{fact5} are true for any Lipschitz continuous function $h:[0,\infty)\to\mathbb{R}$ when $u$ is BV. When instead $u$ is only known to have strong traces (\Cref{strong_trace_definition}), then \eqref{fact4} and \eqref{fact5} are given in Lemma 6 in \cite{Leger2011}. We do not prove \eqref{fact4} and \eqref{fact5}  here; their proof is in the appendix in \cite{Leger2011}.

Note that $V$ (see \eqref{V_def_systems}) is upper semi-continuous in $u$ because indicator functions of open sets are lower semi-continuous and the negative of a lower semi-continuous function is upper semi-continuous.

\uline{Step 3}

Let  $u_{\pm}\coloneqq u(u(h(t)\pm,t)$.

Note that by \Cref{Filippov_existence}, 
\begin{align}\label{where_h_lives}
\dot{h}(t)\in I \Bigg[\lambda_{1}(u_+)-C_*\mathbbm{1}_{\{u|a\eta(u|\bar{u}_+(t))<\eta(u|\bar{u}_-(t))\}}(u_+),\\
\lambda_{1}(u_-)-C_*\mathbbm{1}_{\{u|a\eta(u|\bar{u}_+(t))<\eta(u|\bar{u}_-(t))\}}(u_-)\Bigg].
\end{align}

We are now ready to show \eqref{dissipation_negative_claim}. 

For each fixed time $t$, we have 4 cases to consider to prove \eqref{dissipation_negative_claim}:\newline
\emph{Case 1}
\begin{align}
a\eta(u_-|\bar{u}_+(t))<\eta(u_-|\bar{u}_-(t)),\\
a\eta(u_+|\bar{u}_+(t))<\eta(u_+|\bar{u}_-(t)).
\end{align}
\emph{Case 2}
\begin{align}
a\eta(u_-|\bar{u}_+(t))<\eta(u_-|\bar{u}_-(t)),\\
a\eta(u_+|\bar{u}_+(t))\geq\eta(u_+|\bar{u}_-(t)).
\end{align}
\emph{Case 3}
\begin{align}
a\eta(u_-|\bar{u}_+(t))\geq\eta(u_-|\bar{u}_-(t)),\\
a\eta(u_+|\bar{u}_+(t))<\eta(u_+|\bar{u}_-(t)).
\end{align}
\emph{Case 4}
\begin{align}
a\eta(u_-|\bar{u}_+(t))\geq\eta(u_-|\bar{u}_-(t)),\\
a\eta(u_+|\bar{u}_+(t))\geq\eta(u_+|\bar{u}_-(t)).
\end{align}

Note that we allow for $u_+=u_-$.

We start with 

\emph{Case 1}

In this case, by \eqref{fact3}, \eqref{C_star_def}, and \eqref{where_h_lives} we know that 
\begin{equation}
\begin{aligned}\label{control_h}
\dot{h}(t)\leq -\frac{1}{c_4\gamma_0^2}\Bigg(\sup_{u,u_L,u_R\in B_{B}(0)}\abs{aq(u;u_R)-q(u;u_L)}+1\Bigg)-\sup_{u\in B_{B}(0)}\abs{\lambda_1(u)}
\\
<\inf_{u\in B_{B}(0)} \lambda_1(u).
\end{aligned} 
\end{equation}

If $u_+\neq u_-$, then we have \eqref{fact4} and \eqref{fact5}. But then, \eqref{control_h} contradicts $(\mathcal{H}2)$. Thus, $u_+= u_-$.

Let $v\coloneqq u_+=u_-$.

If $\mbox{dist}(v,\{w|\eta(w|\bar{u}_-(t))\leq a\eta(w|\bar{u}_+(t))\})\geq \gamma_0$, then

\begin{equation}
\begin{aligned}\label{dissipation_negative_claim_proof_case1_1}
&a\bigg(q(u_+;\bar{u}_+(t))-\dot{h}(t)\eta(u_+|\bar{u}_+(t))\bigg)-q(u_-;\bar{u}_-(t))+\dot{h}(t)\eta(u_-|\bar{u}_-(t)) \\
&=a\bigg(q(v;\bar{u}_+(t))-\dot{h}(t)\eta(v|\bar{u}_+(t))\bigg)-q(v;\bar{u}_-(t))+\dot{h}(t)\eta(v|\bar{u}_-(t))\\
&=aq(v;\bar{u}_+(t))-q(v;\bar{u}_-(t))-\dot{h}(t)\big(a\eta(v|\bar{u}_+(t))-\eta(v|\bar{u}_-(t))\big)\\
&\leq -1,
\end{aligned}
\end{equation}
because of \eqref{control_h} and \eqref{bound_on_inf}. Because the term $\abs{\dot{s}(t)-\dot{h}(t)}^2$ on the right hand side of \eqref{dissipation_negative_claim} is bounded due to \eqref{fact1} and $s$ being Lipschitz, we have proven \eqref{dissipation_negative_claim} by choosing $c$ sufficiently small.

If on the other hand, $\mbox{dist}(v,\{w|\eta(w|\bar{u}_-(t))\leq a\eta(w|\bar{u}_+(t))\})< \gamma_0$, then

\begin{equation}
\begin{aligned}\label{dissipation_negative_claim_proof_case1_2}
&a\bigg(q(u_+;\bar{u}_+(t))-\dot{h}(t)\eta(u_+|\bar{u}_+(t))\bigg)-q(u_-;\bar{u}_-(t))+\dot{h}(t)\eta(u_-|\bar{u}_-(t)) \\
&=a\bigg(q(v;\bar{u}_+(t))-\dot{h}(t)\eta(v|\bar{u}_+(t))\bigg)-q(v;\bar{u}_-(t))+\dot{h}(t)\eta(v|\bar{u}_-(t)) \\
&=aq(v;\bar{u}_+(t))-q(v;\bar{u}_-(t))-\dot{h}(t)\big(a\eta(v|\bar{u}_+(t))-\eta(v|\bar{u}_-(t))\big)\\
&\leq a\bigg(q(v;\bar{u}_+(t))-\lambda_1(v)\eta(v|\bar{u}_+(t))\bigg)-q(v;\bar{u}_-(t))+\lambda_1(v)\eta(v|\bar{u}_-(t)),\\
&\mbox{because $\eta(v|\bar{u}_-(t))-a\eta(v|\bar{u}_+(t))\geq 0$ and $\dot{h}\leq -\sup_{u\in B_{B}(0)}\abs{\lambda_1(u)}$. Continuing,} \\
&\mbox{we get}
\\
&\leq -\frac{1}{2}c_1,
\end{aligned}
\end{equation}
from \eqref{dissipation_negative_boundary_of_convex_set}, the definition of $\gamma_0$ \eqref{epsilon_0_def}, the assumption that 
\begin{align}
\mbox{dist}(v,\{w|\eta(w|\bar{u}_-(t))\leq a\eta(w|\bar{u}_+(t))\})< \gamma_0
\end{align}
 and the assumption that $r(t)>\rho$ for all $t$, where $r(t)$ satisfies $S^1_{\bar{u}_-(t)}(r(t))=\bar{u}_+(t)$. Again because the term $\abs{\dot{s}(t)-\dot{h}(t)}^2$ on the right hand side of \eqref{dissipation_negative_claim} is bounded due to \eqref{fact1} and $s$ being Lipschitz, we have proven \eqref{dissipation_negative_claim} by choosing $c$ sufficiently small. Note $c$ will depend on $\rho$.

\emph{Case 2}

In this case, we must have $u_-\neq u_+$. Recall also that \eqref{system} is hyperbolic. Furthermore, we have from \eqref{fact3} that $\dot{h}\in \Bigg[-\frac{1}{c_4\gamma_0^2}\Bigg(\sup_{u,u_L,u_R\in B_{B}(0)}\abs{aq(u;u_R)-q(u;u_L)}+1\Bigg)-\sup_{u\in B_{B}(0)}\abs{\lambda_1(u)},\lambda_1(u_+)\Bigg]$. However, this implies that $(u_+,u_-,\dot{h})$ is a right 1-contact discontinuity (see \cite[p.~274]{dafermos_big_book}). This contradicts the hypothesis $(\mathcal{H}2)$ on the shock $(u_+,u_-,\dot{h})$, which is entropic for $\eta$ because of \eqref{fact4} and \eqref{fact5}. The hypothesis $(\mathcal{H}2)$ forbids right 1-contact discontinuities. Thus, we conclude that this case (\emph{Case 2}) cannot actually occur.

\emph{Case 3}

In this case, we have from \eqref{fact3} that 
\begin{align}
\dot{h}\in \Bigg[-\frac{1}{c_4\gamma_0^2}\Bigg(\sup_{u,u_L,u_R\in B_{B}(0)}\abs{aq(u;u_R)-q(u;u_L)}+1\Bigg)-\sup_{u\in B_{B}(0)}\abs{\lambda_1(u)},\lambda_1(u_-)\Bigg].
\end{align}
By the hypothesis $(\mathcal{H}3)$, along with \eqref{fact4}, \eqref{fact5}, we have that $(u_+,u_-,\dot{h})$ must be a 1-shock. Also, $u_-$ verifies $a\eta(u_-|\bar{u}_+(t))\geq\eta(u_-|\bar{u}_-(t))$. Thus, we can apply \Cref{dissipation_negative_theorem}. Recall that $r(t)>\rho$ for all $t$, where $r(t)$ satisfies $S^1_{\bar{u}_-(t)}(r(t))=\bar{u}_+(t)$. We receive  \eqref{dissipation_negative_claim}.

\emph{Case 4}

In this case, we have from \eqref{fact3} that $\dot{h}\in I[\lambda_1(u_+),\lambda_1(u_-)]$. Then, by the hypothesis $(\mathcal{H}2)$, along with \eqref{fact4}, \eqref{fact5}, we know that we cannot have
\begin{align}\label{this_would_imply_bad}
I[\lambda_1(u_+),\lambda_1(u_-)]=(\lambda_1(u_-),\lambda_1(u_+))
\end{align}
because then \eqref{this_would_imply_bad} would imply that $(u_+,u_-,\dot{h})$ is a right 1-contact discontinuity. However, $(\mathcal{H}2)$ prevents right 1-contact discontinuities. Recall $(\mathcal{H}3)$. We conclude that $(u_+,u_-,\dot{h})$ is a 1-shock. Moreover, $u_-$ verifies $a\eta(u_-|\bar{u}_+(t))\geq\eta(u_-|\bar{u}_-(t))$. We can now apply \Cref{dissipation_negative_theorem}. Recall that $r(t)>\rho$ for all $t$, where $r(t)$ satisfies $S^1_{\bar{u}_-(t)}(r(t))=\bar{u}_+(t)$. This gives \eqref{dissipation_negative_claim}.

\section{Proof of main theorem \Cref{local_stability_systems}}\label{proof_main_theorem}

Note that if $\bar{u}$ contains an n-shock, then  the solution $(x,t)\mapsto \bar{u}(-x,t)$ to the system $\partial_t u -f(u)=G(u)$ will have 1-shock for this system. Thus, we can always assume $\bar{u}$ has a 1-shock.

Let $h$ be as in \Cref{systems_entropy_dissipation_room}. 

Define
\begin{equation}
\begin{aligned}\label{h_defs}
h_1(t)\coloneqq -R+s(0)+r(t-t_0),\\
h_2(t)\coloneqq R+s(0)-r(t-t_0),
\end{aligned}
\end{equation}
where $r>0$ verifies
\begin{align}\label{r_def}
\abs{q(u;\bar{u})}\leq r \eta(u|\bar{u}).
\end{align}
Such an $r>0$ exists because $u$ and $\bar{u}$ are bounded, $q(a;b)$ and $\eta(a|b)$ are both locally quadratic in $a-b$, and $\eta$ is strictly convex.

Then we apply \Cref{local_entropy_dissipation_rate_systems} to $h_1$ and $h$. This yields,
\begin{equation}
\begin{aligned}\label{local_compatible_dissipation_calc_left}
\int\limits_{0}^{t_0} \bigg[q(u(h_1(t)+,t);\bar{u}((h_1(t)+X(t))+,t))-q(u(h(t)-,t);\bar{u}((h(t)+X(t))-,t))\hspace{.5in}
\\
+\dot{h}(t)\eta(u(h(t)-,t)|\bar{u}((h(t)+X(t))-,t))\hspace{2in}
\\
-\dot{h}_1(t)\eta(u(h_1(t)+,t)|\bar{u}((h_1(t)+X(t))+,t))\bigg]\,dt\hspace{1.57in}
\\
\geq
\int\limits_{h_1(t_0)}^{h(t_0)}\eta(u(x,t_0)|\bar{u}(x+X(t_0),t_0))\,dx
-\int\limits_{h_1(0)}^{h(0)}\eta(u^0(x)|\bar{u}^0(x))\,dx
\\
+\int\limits_{0}^{t_0}\int\limits_{h_1(t)}^{h(t)}\Bigg(\partial_x \bigg|_{(x+X(t),t)}\hspace{-.45in} \bar{u}^T(x,t)\Bigg)\nabla^2\eta(\bar{u}(x+X(t),t)) f(u(x,t)|\bar{u}(x+X(t),t))
\\
+\Bigg(2\partial_x\bigg|_{(x+X(t),t)}\hspace{-.45in}\bar{u}^T(x,t)\dot{X}(t)\Bigg)\nabla^2\eta(\bar{u}(x+X(t),t))[u(x,t)-\bar{u}(x+X(t),t)]
\\
-\nabla\eta(u(x,t)|\bar{u}(x+X(t),t))G(u(\cdot,t))(x)
\\
+
\Bigg(G(\bar{u}(\cdot,t))(x+X(t))-G(u(\cdot,t))(x)\Bigg)^T\nabla^2\eta(\bar{u}(x+X(t),t))[u(x,t)-\bar{u}(x+X(t),t)]\,dxdt,
\end{aligned}
\end{equation}
where
\begin{equation}
\begin{aligned}\label{defs_z_X}
f(u|\bar{u})\coloneqq f(u)-f(\bar{u})-\nabla f (\bar{u})(u-\bar{u}),\\
X(t)\coloneqq s(t)-h(t).
\end{aligned}
\end{equation}
Similarly, we apply \Cref{local_entropy_dissipation_rate_systems} to $h$ and $h_2$. This yields,
\begin{equation}
\begin{aligned}\label{local_compatible_dissipation_calc_right}
\int\limits_{0}^{t_0} \bigg[q(u(h(t)+,t);\bar{u}((h(t)+X(t))+,t))-q(u(h_2(t)-,t);\bar{u}((h_2(t)+X(t))-,t))\hspace{.5in}
\\
+\dot{h}_2(t)\eta(u(h_2(t)-,t)|\bar{u}((h_2(t)+X(t))-,t))\hspace{2in}
\\
-\dot{h}(t)\eta(u(h(t)+,t)|\bar{u}((h(t)+X(t))+,t))\bigg]\,dt\hspace{1.95in}
\\
\geq
\int\limits_{h(t_0)}^{h_2(t_0)}\eta(u(x,t_0)|\bar{u}(x+X(t_0),t_0))\,dx
-\int\limits_{h(0)}^{h_2(0)}\eta(u^0(x)|\bar{u}^0(x))\,dx
\\
+\int\limits_{0}^{t_0}\int\limits_{h(t)}^{h_2(t)}\Bigg(\partial_x \bigg|_{(x+X(t),t)}\hspace{-.45in} \bar{u}^T(x,t)\Bigg)\nabla^2\eta(\bar{u}(x+X(t),t)) f(u(x,t)|\bar{u}(x+X(t),t))
\\
+\Bigg(2\partial_x\bigg|_{(x+X(t),t)}\hspace{-.45in}\bar{u}^T(x,t)\dot{X}(t)\Bigg)\nabla^2\eta(\bar{u}(x+X(t),t))[u(x,t)-\bar{u}(x+X(t),t)]
\\
-\nabla\eta(u(x,t)|\bar{u}(x+X(t),t))G(u(\cdot,t))(x)
\\
+
\Bigg(G(\bar{u}(\cdot,t))(x+X(t))-G(u(\cdot,t))(x)\Bigg)^T\nabla^2\eta(\bar{u}(x+X(t),t))[u(x,t)-\bar{u}(x+X(t),t)]\,dxdt.
\end{aligned}
\end{equation}

We combine \eqref{local_compatible_dissipation_calc_left} and $a$ multiples of \eqref{local_compatible_dissipation_calc_right}. This gives,

\begin{equation}
\begin{aligned}\label{local_compatible_dissipation_calc_left_and_right}
\int\limits_{0}^{t_0} \bigg[a\Bigg(q(u(h(t)+,t);\bar{u}((h(t)+X(t))+,t))-\dot{h}(t)\eta(u(h(t)+,t)|\bar{u}((h(t)+X(t))+,t))\Bigg)
\\
+\dot{h}(t)\eta(u(h(t)-,t)|\bar{u}((h(t)+X(t))-,t))
-q(u(h(t)-,t);\bar{u}((h(t)+X(t))-,t))
\\
+aq(u(h_1(t)+,t);\bar{u}((h_1(t)+X(t))+,t))
-a\dot{h}_1(t)\eta(u(h_1(t)+,t)|\bar{u}((h_1(t)+X(t))+,t))
\\
-q(u(h_2(t)-,t);\bar{u}((h_2(t)+X(t))-,t))
+\dot{h}_2(t)\eta(u(h_2(t)-,t)|\bar{u}((h_2(t)+X(t))-,t))\bigg]\,dt
\\
\geq
\Bigg[a\int\limits_{h_1(t_0)}^{h(t_0)}\eta(u(x,t_0)|\bar{u}(x+X(t_0),t_0))\,dx+\int\limits_{h(t_0)}^{h_2(t_0)}\eta(u(x,t_0)|\bar{u}(x+X(t_0),t_0))\,dx\Bigg]
\\
-\Bigg[a\int\limits_{h_1(0)}^{h(0)}\eta(u^0(x)|\bar{u}^0(x))\,dx+\int\limits_{h(0)}^{h_2(0)}\eta(u^0(x)|\bar{u}^0(x))\,dx\Bigg]
\\
+\int\limits_{0}^{t_0}\int\limits_{\mathbb{R}}\mathbbm{1}_a(x)\Bigg[\Bigg(\partial_x \bigg|_{(x+X(t),t)}\hspace{-.45in} \bar{u}^T(x,t)\Bigg)\nabla^2\eta(\bar{u}(x+X(t),t)) f(u(x,t)|\bar{u}(x+X(t),t))
\\
+\Bigg(2\partial_x\bigg|_{(x+X(t),t)}\hspace{-.45in}\bar{u}^T(x,t)\dot{X}(t)\Bigg)\nabla^2\eta(\bar{u}(x+X(t),t))[u(x,t)-\bar{u}(x+X(t),t)]
-
\\
\nabla\eta(u(x,t)|\bar{u}(x+X(t),t))G(u(\cdot,t))(x)
\\
+
\Bigg(G(\bar{u}(\cdot,t))(x+X(t))-G(u(\cdot,t))(x)\Bigg)^T\nabla^2\eta(\bar{u}(x+X(t),t))[u(x,t)-\bar{u}(x+X(t),t)]\Bigg]\,dxdt,
\end{aligned}
\end{equation}
where 
\begin{align}
\mathbbm{1}_a(x)\coloneqq a\mathbbm{1}_{\{x|h_1(t)<x<h(t)\}}(x)+\mathbbm{1}_{\{x|h(t)<x<h_2(t)\}}(x).
\end{align}

We estimate the last term on the right hand side of \eqref{local_compatible_dissipation_calc_left_and_right}, which is of the form
\begin{align}
\int\limits_{0}^{t_0}\int\limits_{\mathbb{R}}\overbracket[.5pt][7pt]{\mathbbm{1}_a(x)}^{L^\infty(\mathbb{R})}\overbracket[.5pt][7pt]{\Bigg[\cdots\Bigg]}^{L^1([h_1(t),h_2(t)])}\,dxdt,
\end{align}
using the indicated H\"older dualities.

We then want to estimate from above the term 
\begin{equation}
\begin{aligned}\label{big_term_estimate_above}
&\int\limits_{h_1(t)}^{h_2(t)}\Bigg|\overbracket[.5pt][7pt]{\Bigg(\partial_x \bigg|_{(x+X(t),t)}\hspace{-.45in} \bar{u}^T(x,t)\Bigg)\nabla^2\eta(\bar{u}(x+X(t),t))}^{L^\infty([h_1(t),h_2(t)])}\overbracket[.5pt][7pt]{f(u(x,t)|\bar{u}(x+X(t),t))}^{L^1([h_1(t),h_2(t)])}
\\
&+\dot{X}(t)\overbracket[.5pt][7pt]{\Bigg(2\partial_x\bigg|_{(x+X(t),t)}\hspace{-.45in}\bar{u}^T(x,t)\Bigg)}^{L^2([h_1(t),h_2(t)])}\overbracket[.5pt][7pt]{\nabla^2\eta(\bar{u}(x+X(t),t))}^{L^\infty([h_1(t),h_2(t)])}\overbracket[.5pt][7pt]{[u(x,t)-\bar{u}(x+X(t),t)]}^{L^2([h_1(t),h_2(t)])}
\\
&-\overbracket[.5pt][7pt]{\nabla\eta(u(x,t)|\bar{u}(x+X(t),t))}^{L^1([h_1(t),h_2(t)])}\overbracket[.5pt][7pt]{G(u(\cdot,t))(x)}^{L^\infty([h_1(t),h_2(t)])}
\\
&+
\overbracket[.5pt][7pt]{\Bigg(G(\bar{u}(\cdot,t))(x+X(t))-G(u(\cdot,t))(x)\Bigg)^T}^{L^2([h_1(t),h_2(t)])}\overbracket[.5pt][7pt]{\nabla^2\eta(\bar{u}(x+X(t),t))}^{L^\infty([h_1(t),h_2(t)])}\overbracket[.5pt][7pt]{[u(x,t)-\bar{u}(x+X(t),t)]}^{L^2([h_1(t),h_2(t)])}\Bigg|\,dx.
\end{aligned}
\end{equation}

We use the H\"older dualities indicated above. In particular, recall that $f(a|b)$ is locally quadratic in  $a-b$ and that $\partial_x \bar{u}\in L^\infty(\mathbb{R}\times[0,T))$ due to $\bar{u}$ being Lipschitz continuous.

Note that from $G:(L^2(\mathbb{R}))^n\to (L^2(\mathbb{R}))^n$ being translation invariant and from \eqref{G_acts_like}, we have
\begin{equation}
\begin{aligned}\label{control_on_difference_G}
&\norm{G(\bar{u}(\cdot,t))(\cdot+X(t))-G(u(\cdot,t))(\cdot)}_{L^2([h_1(t),h_2(t)]} \\
&\hspace{1in}=\norm{G(\bar{u}(\cdot+X(t),t))(\cdot)-G(u(\cdot,t))(\cdot)}_{L^2([h_1(t),h_2(t)])}
\\
&\hspace{1in}\leq
C_G\norm{\bar{u}(\cdot+X(t),t)-u(\cdot,t)}_{L^2([h_1(t),h_2(t)])},
\end{aligned}
\end{equation}

where $C_G$ is from \eqref{G_acts_like}. 

Recall also \eqref{G_acts_like_2}.

Note also that we can estimate,
\begin{align}\label{estimate_partial_x_u_bar}
\norm{\partial_x\bar{u}(\cdot+X(t),t)}_{L^2([h_1(t),h_2(t)])}
\leq
\sqrt{2(R+rT)}\norm{\partial_x\bar{u}}_{L^\infty(\mathbb{R}\times[0,T))}=\sqrt{2(R+rT)}\mbox{Lip}[\bar{u}].
\end{align}

For $\norm{\partial_x\bar{u}(\cdot+X(t),t)}_{L^2([h_1(t),h_2(t)])}\norm{\nabla^2\eta(\bar{u})}_{L^\infty}\neq0$ we have, from using the `Young's inequality with $\epsilon$,'
\begin{equation}
\begin{aligned}\label{youngs_inequality_part1}
&\abs{\dot{X}(t)}\norm{u(\cdot,t)-\bar{u}(\cdot+X(t),t)}_{L^2([h_1(t),h_2(t)])}  
\\
&\hspace{.3in}\leq\frac{c}{4\norm{\partial_x\bar{u}(\cdot+X(t),t)}_{L^2([h_1(t),h_2(t)])}\norm{\nabla^2\eta(\bar{u})}_{L^\infty}}(\dot{X}(t))^2
\\
&\hspace{.3in}+\frac{\norm{\partial_x\bar{u}(\cdot+X(t),t)}_{L^2([h_1(t),h_2(t)])}\norm{\nabla^2\eta(\bar{u})}_{L^\infty}}{c} \norm{u(\cdot,t)-\bar{u}(\cdot+X(t),t)}_{L^2([h_1(t),h_2(t)])}^2,
\end{aligned}
\end{equation}
where $c$ is from the right hand side of \eqref{dissipation_negative_claim}. Note that $c$ depends on $\rho$, $\norm{u}_{L^\infty}$, $\norm{\bar{u}(s(t)+,t)}_{L^\infty([0,T))}$, $\norm{\bar{u}(s(t)-,t)}_{L^\infty([0,T))}$, and $a$. From \eqref{youngs_inequality_part1}, we get
\begin{equation}
\begin{aligned}\label{youngs_inequality}
&2\abs{\dot{X}(t)}\norm{\partial_x\bar{u}(\cdot+X(t),t)}_{L^2([h_1(t),h_2(t)])}\norm{\nabla^2\eta(\bar{u})}_{L^\infty}\norm{u(\cdot,t)-\bar{u}(\cdot+X(t),t)}_{L^2([h_1(t),h_2(t)])}
\\
&\leq\frac{c}{2}(\dot{X}(t))^2+\frac{2\norm{\partial_x\bar{u}(\cdot+X(t),t)}_{L^2([h_1(t),h_2(t)])}^2\norm{\nabla^2\eta(\bar{u})}_{L^\infty}^2}{c} \norm{u(\cdot,t)-\bar{u}(\cdot+X(t),t)}_{L^2([h_1(t),h_2(t)])}^2.
\end{aligned}
\end{equation}
If for some $t$, $\norm{\partial_x\bar{u}(\cdot+X(t),t)}_{L^2([h_1(t),h_2(t)])}\norm{\nabla^2\eta(\bar{u})}_{L^\infty}=0$, then we don't have to estimate the term
\begin{align}
\dot{X}(t)\Bigg(2\partial_x\bigg|_{(x+X(t),t)}\hspace{-.45in}\bar{u}^T(x,t)\Bigg)\nabla^2\eta(\bar{u}(x+X(t),t))[u(x,t)-\bar{u}(x+X(t),t)].
\end{align}

Recall \eqref{h_defs} and \eqref{r_def}. Note in particular we have $\dot{h}_1=r$ and $\dot{h}_2=-r$. Then from \eqref{dissipation_negative_claim} (in \Cref{systems_entropy_dissipation_room}) and \eqref{youngs_inequality}, we get

\begin{equation}
\begin{aligned}\label{local_compatible_dissipation_calc_left_and_right_one_piece}
&-\int\limits_{0}^{t_0}\int\limits_{\mathbb{R}}\Bigg[
\Bigg(2\partial_x\bigg|_{(x+X(t),t)}\hspace{-.45in}\bar{u}^T(x,t)\dot{X}(t)\Bigg)\nabla^2\eta(\bar{u}(x+X(t),t))[u(x,t)-\bar{u}(x+X(t),t)]
\Bigg]\,dxdt
\\
&\hspace{.5in}+
\int\limits_{0}^{t_0} \bigg[a\Bigg(q(u(h(t)+,t);\bar{u}((h(t)+X(t))+,t))-\dot{h}(t)\eta(u(h(t)+,t)|\bar{u}((h(t)+X(t))+,t))\Bigg)
\\
&\hspace{.5in}+\dot{h}(t)\eta(u(h(t)-,t)|\bar{u}((h(t)+X(t))-,t))
-q(u(h(t)-,t);\bar{u}((h(t)+X(t))-,t))
\\
&\hspace{.5in}+aq(u(h_1(t)+,t);\bar{u}((h_1(t)+X(t))+,t))
-a\dot{h}_1(t)\eta(u(h_1(t)+,t)|\bar{u}((h_1(t)+X(t))+,t))
\\
&\hspace{.5in}-q(u(h_2(t)-,t);\bar{u}((h_2(t)+X(t))-,t))
+\dot{h}_2(t)\eta(u(h_2(t)-,t)|\bar{u}((h_2(t)+X(t))-,t))\bigg]\,dt
\\
&\hspace{.2in}\leq
\int\limits_0^{t_0} -\frac{c}{2}(\dot{X}(t))^2
\\ &\hspace{.2in}+
\frac{2\norm{\partial_x\bar{u}(\cdot+X(t),t)}_{L^2([h_1(t),h_2(t)])}^2\norm{\nabla^2\eta(\bar{u})}_{L^\infty}^2}{c} \norm{u(\cdot,t)-\bar{u}(\cdot+X(t),t)}_{L^2([h_1(t),h_2(t)])}^2\,dt.
\end{aligned}
\end{equation}

Recall \eqref{control_on_difference_G}, \eqref{estimate_partial_x_u_bar}, and \eqref{local_compatible_dissipation_calc_left_and_right_one_piece}. Recall also \eqref{h_defs} and \eqref{defs_z_X}. Further, recall from \Cref{systems_entropy_dissipation_room} that $h(0)=s(0)$. Recall also that from \Cref{systems_entropy_dissipation_room}, we know the constant $c$ depends on $\rho$, $\norm{u}_{L^\infty}$, and $\norm{\bar{u}}_{L^\infty}$. Lastly, recall that $f(a|b)$, $\eta(a|b)$, and $\nabla\eta(a|b)$ are locally quadratic in $a-b$ (recall $\eta\in C^3(\mathbb{R}^n)$), and from the strict convexity of $\eta$ we in fact have \Cref{entropy_relative_L2_control_system}. Then, from \eqref{local_compatible_dissipation_calc_left_and_right}, we receive

\begin{equation}
\begin{aligned}\label{right_before_Gronwall}
\mu_1\int\limits_0^{t_0}\int\limits_{h_1(t)}^{h_2(t)}\abs{u(x,t)-\bar{u}(x+X(t),t)}^2\,dxdt
+\mu_2\int\limits_{-R-rt_0+s(0)}^{R+rt_0+s(0)}\abs{u^0(x)-\bar{u}^0(x)}^2\,dx
\\
-\frac{1}{\mu_2}\int\limits_0^{t_0} (\dot{X}(t))^2\,dt
\geq 
\int\limits_{-R+s(0)}^{R+s(0)}\abs{u(x,t_0)-\bar{u}(x+X(t_0),t_0)}^2\,dx
\end{aligned}
\end{equation}
for all $t_0\in[0,T)$, where $\mu_1,\mu_2>0$ are constants depending on $a$, $\rho$, $\norm{u}_{L^\infty}$, $\norm{\bar{u}}_{L^\infty}$, and bounds on the derivatives of $\eta$ on the range of $u$ and $\bar{u}$.    Furthermore,  $\mu_1$ also depends on $C_G$ (see \eqref{G_acts_like} and \eqref{G_acts_like_2}), $\mbox{Lip}[\bar{u}]$, $\rho$, $R$, $T$, and bounds on the derivatives of $f$ on the range of $u$ and $\bar{u}$. Note that $r$ (see \eqref{r_def}) only depends on bounds on the derivatives of $f$ and $\eta$ on the (range of $u$ and $\bar{u}$). The constant $a$ then itself depends on $\rho$, $\norm{u}_{L^\infty}$, and $\norm{\bar{u}}_{L^\infty}$ (see \Cref{dissipation_negative_theorem}).

We can drop the last term on the left hand side of \eqref{right_before_Gronwall}, to get 
\begin{equation}
\begin{aligned}\label{right_before_Gronwall462019}
\mu_1\int\limits_0^{t_0}\int\limits_{h_1(t)}^{h_2(t)}\abs{u(x,t)-\bar{u}(x+X(t),t)}^2\,dxdt
+\mu_2\int\limits_{-R-rt_0+s(0)}^{R+rt_0+s(0)}\abs{u^0(x)-\bar{u}^0(x)}^2\,dx
\\
\geq 
\int\limits_{-R+s(0)}^{R+s(0)}\abs{u(x,t_0)-\bar{u}(x+X(t_0),t_0)}^2\,dx.
\end{aligned}
\end{equation}

We then apply the Gronwall inequality to \eqref{right_before_Gronwall462019}. This yields,
\begin{align}\label{Gronwall_in_proof_piecewise_systems}
&\int\limits_{-R+s(0)}^{R+s(0)}\abs{u(x,t_0)-\bar{u}(x+X(t_0),t_0)}^2\,dx
\\
&\hspace{1in}\leq \mu_2 e^{\mu_1 t_0}\Bigg(\int\limits_{-R-rt_0+s(0)}^{R+rt_0+s(0)}\abs{u^0(x)-\bar{u}^0(x)}^2\,dx\Bigg).
\end{align}

From \eqref{Gronwall_in_proof_piecewise_systems}, we get  \eqref{main_local_stability_result}.

We now show \eqref{L2_control_shift_piecewise_systems}. From \eqref{right_before_Gronwall}, we get
\begin{equation}
\begin{aligned}\label{right_before_Gronwall_for_control_462019}
&\mu_1\int\limits_0^{t_0}\int\limits_{h_1(t)}^{h_2(t)}\abs{u(x,t)-\bar{u}(x+X(t),t)}^2\,dxdt
+\mu_2\int\limits_{-R-rt_0+s(0)}^{R+rt_0+s(0)}\abs{u^0(x)-\bar{u}^0(x)}^2\,dx
\\
&\hspace{3in}\geq 
\frac{1}{\mu_2}\int\limits_0^{t_0} (\dot{X}(t))^2\,dt.
\end{aligned}
\end{equation}
Then we bootstrap, and use \eqref{main_local_stability_result} to estimate the term 
\begin{align*}
\int\limits_{h_1(t)}^{h_2(t)}\abs{u(x,t)-\bar{u}(x+X(t),t)}^2\,dx
\end{align*}
in \eqref{right_before_Gronwall_for_control_462019}. This gives  \eqref{L2_control_shift_piecewise_systems}.

This proves \Cref{local_stability_systems}.

\section{Appendix}
\subsection{Proof of \Cref{a_cond_lemma_itself}}\label{appendix_a_cond_lemma_itself}

Throughout this proof, $C$ will denote a generic constant depending only on $B$.

We will first show that for $0<a<1$, the set $R_a$ is convex.

For $a<1$, we can rewrite
\begin{align}
\eta(u|u_L)\leq a\eta(u|u_R)
\end{align}
as 
\begin{align}
\label{convex_rewrite}
\eta(u)\leq \frac{1}{1-a}(\eta(u_L)-a\eta(u_R)-\nabla\eta(u_L)\cdot u_L +a\nabla\eta(u_R)\cdot u_R+(\nabla\eta(u_L)-a\nabla\eta(u_R))\cdot u).
\end{align}
The right hand side of \eqref{convex_rewrite} is (affine) linear in $u$. Thus the convexity of $\eta$ implies that $R_a= \{u | \eta(u|u_L)\leq a\eta(u|u_R)\}$ is convex. 

For $a<\frac{1}{2}$, we can rewrite  \eqref{convex_rewrite} to get
\begin{align}
\eta(u|u_L)&\leq \frac{a}{1-a}(\eta(u_L)-\eta(u_R)-\nabla\eta(u_L)\cdot u_L +\nabla\eta(u_R)\cdot u_R+(\nabla\eta(u_L)-\nabla\eta(u_R))\cdot u)\\
&\leq Ca(1+\abs{u}).
\end{align}

We combine this with \Cref{entropy_relative_L2_control_system} to get that for all $u\in R_a\cap B_{\theta}(u_L)$ (recalling $\theta<1$),
\begin{align}
\label{proto_cond_a}
\abs{u-u_L}^2\leq Ca(1+\abs{u}) \leq Ca.
\end{align}

Thus, when $\alpha$ satisfies \eqref{cond_a} with $C$ as in \eqref{proto_cond_a}, and $0<a<\alpha$, we have
\begin{align}
\abs{u-u_L}^2\leq Ca < \frac{\theta^2}{2}.
\end{align}
Thus $R_a\cap B_{\theta}(u_L)$ is strictly contained in $B_{\theta}(u_L)$. As we have shown, the set $R_a$ is convex. Thus $R_a$ is also connected, which implies that 
\begin{align}
R_a=R_a\cap B_{\theta}(u_L).
\end{align}
We conclude that $R_a\subset B_{\theta}(u_L)$ for all $0<a<\alpha$. This completes the proof.

\subsection{Proof of \Cref{Filippov_existence}}\label{Filippov_existence_section}
The following proof of \eqref{fact1}, \eqref{fact2}, and \eqref{fact3} is based on the proof of Proposition 1 in \cite{Leger2011}, the proof of Lemma 2.2 in \cite{serre_vasseur}, and the proof of Lemma 3.5 in \cite{2017arXiv170905610K}. We do not prove \eqref{fact4} or \eqref{fact5} here; these properties are in Lemma 6 in \cite{Leger2011}, and their proofs are in the appendix in \cite{Leger2011}.

Define

\begin{align}
v_n(x,t)\coloneqq \int\limits_0^1 V\bigg(u(x+\frac{y}{n},t),t\bigg)\,dy.
\end{align}

Let $h_{n}$ be the solution to the ODE:
\begin{align}
  \begin{cases}\label{n_ode}
   \dot{h}_n(t)=v_n(h_n(t),t),\mbox{ for }t>0\\
   h_n(0)=x_0.
  \end{cases}
\end{align}

The $v_n$ are uniformly bounded in $n$ because by assumption $V$ is bounded (\hspace{.07cm}$\norm{v_n}_{L^\infty}\leq \norm{V}_{L^\infty}$). The $v_n$ are measurable in $t$, and due to the mollification by $\frac{1}{n}$ are also Lipschitz continuous in $x$. Thus \eqref{n_ode} has a unique solution in the sense of Carath\'eodory.

The $h_n$ are Lipschitz continuous with Lipschitz constants uniform in $n$, due to the $v_n$ being uniformly bounded in $n$. Thus, by Arzel\`a--Ascoli the $h_n$ converge in $C^0(0,T)$ for any fixed $T>0$ to a Lipschitz continuous function $h$ (passing to a subsequence if necessary). Note that $\dot{h}_n$ converges in $L^\infty$ weak* to $\dot{h}$.

We define
\begin{align}
V_{\mbox{max}}(t)\coloneqq \max\{V(u_-,t),V(u_+,t)\},\\
V_{\mbox{min}}(t)\coloneqq \min\{V(u_-,t),V(u_+,t)\},
\end{align}
where $u_\pm \coloneqq u(h(t)\pm,t)$.

To show \eqref{fact3}, we will first prove that for almost every $t>0$
\begin{align}
\lim_{n\to\infty}[\dot{h}_n(t)-V_{\mbox{max}}(t)]_+=0,\label{limit_1}\\
\lim_{n\to\infty}[V_{\mbox{min}}(t)-\dot{h}_n(t)]_+=0,\label{limit_2}
\end{align}
where $[\hspace{.1cm}\cdot\hspace{.1cm}]_+\coloneqq\max(0,\cdot)$.

The proofs of \eqref{limit_1} and \eqref{limit_2} are similar; we only show the first one.

\begin{align}
[\dot{h}_n(t)-V_{\mbox{max}}(t)]_+\\
=\Bigg[\int\limits_0^1 V\bigg(u(h_n(t)+\frac{y}{n},t),t\bigg)\,dy-V_{\mbox{max}}(t)\Bigg]_+\\
=\Bigg[\int\limits_0^1 V\bigg(u(h_n(t)+\frac{y}{n},t),t\bigg)-V_{\mbox{max}}(t)\,dy\Bigg]_+\\
\leq\int\limits_0^1 \Big[V\bigg(u(h_n(t)+\frac{y}{n},t),t\bigg)-V_{\mbox{max}}(t)\Big]_+\,dy\\
\leq\esssup_{y\in(0,\frac{1}{n})} \Big[V\bigg(u(h_n(t)+y,t),t\bigg)-V_{\mbox{max}}(t)\Big]_+\\
\leq\esssup_{y\in(-\epsilon_n,\epsilon_n)} \Big[V\bigg(u(h(t)+y,t),t\bigg)-V_{\mbox{max}}(t)\Big]_+,\label{last_ineq_Filippov}
\end{align}
where $\epsilon_n\coloneqq \abs{h_n(t)-h(t)}+\frac{1}{n}$. Note $\epsilon_n\to0^+$.

Fix a $t\geq0$ such that $u$ has a strong trace in the sense of \Cref{strong_trace_definition}. Then because the map $u\mapsto V(u,t)$ is upper semi-continuous,
\begin{align}\label{esssuplim_is_zero}
\lim_{n\to\infty}\esssup_{y\in(0,\frac{1}{n})} \Big[V\bigg(u(h(t)\pm y,t),t\bigg)-V\big(u_\pm,t\big)\Big]_+=0,
\end{align}
where $u_\pm \coloneqq u(h(t)\pm,t)$. Recall that the map $u\mapsto V(u,t)$ being upper semi-continuous at the point $u_0$ means that 
\begin{align}
\limsup_{u\to u_0} V(u,t) \leq V(u_0,t).
\end{align}

From \eqref{esssuplim_is_zero}, we get
\begin{align}\label{esssuplim_is_zero2}
\lim_{n\to\infty}\esssup_{y\in(0,\frac{1}{n})} \Big[V\bigg(u(h(t)\pm y,t),t\bigg)-V_{\mbox{max}}(t)\Big]_+=0.
\end{align}

We can control \eqref{last_ineq_Filippov} from above by the quantity
\begin{equation}
\begin{aligned}\label{esssuplim_is_zero3}
\esssup_{y\in(-\epsilon_n,0)} \Big[V\bigg(u(h(t)+ y,t),t\bigg)-V_{\mbox{max}}(t)\Big]_++\\
\esssup_{y\in(0,\epsilon_n)} \Big[V\bigg(u(h(t)+ y,t),t\bigg)-V_{\mbox{max}}(t)\Big]_+.
\end{aligned}
\end{equation}

By \eqref{esssuplim_is_zero2}, we have that \eqref{esssuplim_is_zero3} goes to $0$ as $n\to\infty$. This proves \eqref{limit_1}.

Recall that $\dot{h}_n$ converges in $L^\infty$ weak* to $\dot{h}$. Thus, due to the convexity of the function $[\hspace{.1cm}\cdot\hspace{.1cm}]_+$,
\begin{align}
\int\limits_0^T[\dot{h}(t)-V_{\mbox{max}}(t)]_+\,dt\leq \liminf_{n\to\infty}\int\limits_0^T[\dot{h}_n(t)-V_{\mbox{max}}(t)]_+\,dt.
\end{align}

By the dominated convergence theorem and \eqref{limit_1},
\begin{align}
\liminf_{n\to\infty}\int\limits_0^T[\dot{h}_n(t)-V_{\mbox{max}}(t)]_+\,dt=0.
\end{align}

We conclude,
\begin{align}
\int\limits_0^T[\dot{h}(t)-V_{\mbox{max}}(t)]_+\,dt=0.
\end{align}

From a similar argument,
\begin{align}
\int\limits_0^T[V_{\mbox{min}}(t)-\dot{h}(t)]_+\,dt=0.
\end{align}

This proves \eqref{fact3}.

\bibliographystyle{plain}
\bibliography{references}
\end{document}